\documentclass[10.5pt,reqno]{amsart}
\usepackage{longtable} 
\usepackage{hyperref}
\usepackage[T1]{fontenc}
\usepackage[utf8]{inputenc}
\usepackage[english]{babel}
\usepackage{textcomp}
\usepackage{dsfont}
\usepackage{latexsym}
\usepackage{amssymb}
\usepackage{amsthm}
\usepackage{amsmath}
\DeclareMathAlphabet{\mathpzc}{OT1}{pzc}{m}{en}
\usepackage{yfonts}
\usepackage{xfrac}
\usepackage{newlfont}
\usepackage{graphicx}
\usepackage{mathtools}
\usepackage{comment}
\usepackage{indentfirst}
\usepackage{braket}
\usepackage{mathrsfs}
\usepackage{xcolor}

\usepackage{etoolbox}

\usepackage{scalerel}[2014/03/10]
\usepackage[usestackEOL]{stackengine}
\newcommand{\dashint}{\,\ThisStyle{\ensurestackMath{%
			\stackinset{c}{.2\LMpt}{c}{.5\LMpt}{\SavedStyle-}{\SavedStyle\phantom{\int}}}%
		\setbox0=\hbox{$\SavedStyle\int\,$}\kern-\wd0}\int}

\DeclareMathOperator{\pr}{pr}

\DeclareMathOperator{\tr}{Tr}

\DeclareMathOperator{\Hol}{Hol}
\DeclareMathOperator{\Cl}{Cl}

\renewcommand{\Re}{\mathrm{Re}\,}
\renewcommand{\Im}{\mathrm{Im}\,}

\newcommand{\ee}{\mathrm{e}}
\newcommand{\PM}{\mathrm{PM}}

\newcommand{\dd}{\mathrm{d}}

\DeclarePairedDelimiter{\abs}{\lvert}{\rvert}

\DeclarePairedDelimiter{\norm}{\lVert}{\rVert}

\let\originalleft\left
\let\originalright\right
\renewcommand{\left}{\mathopen{}\mathclose\bgroup\originalleft}
\renewcommand{\right}{\aftergroup\egroup\originalright}

\newcommand{\N}{\mathds{N}}

\newcommand{\Db}{\mathds{D}}
\newcommand{\C}{\mathds{C}}

\newcommand{\R}{\mathds{R}}

\newcommand{\T}{\mathds{T}}
\newcommand{\bD}{{\mathrm{b}D}}

\newcommand{\Ff}{\mathfrak{F}}

\newcommand{\Nf}{\mathfrak{N}}

\newcommand{\Uf}{\mathfrak{U}}

\newcommand{\Cc}{\mathcal{C}}

\newcommand{\Fc}{\mathcal{F}}

\newcommand{\Hc}{\mathcal{H}}

\newcommand{\Lc}{\mathcal{L}}
\renewcommand{\Mc}{\mathcal{M}}

\newcommand{\Pc}{\mathcal{P}}

\newcommand{\meg}{\leqslant}
\newcommand{\Meg}{\geqslant}
\newcommand{\eps}{\varepsilon}
\renewcommand{\phi}{\varphi}
\newcommand{\mi}{\mu}

\newcommand{\Lin}{\mathscr{L}}

\title{Clark Measures on Bounded Symmetric Domains}

\begin{document}

\theoremstyle{definition}
\newtheorem{deff}{Definition}[section]

\newtheorem{oss}[deff]{Remark}

\newtheorem{ass}[deff]{Assumptions}

\newtheorem{nott}[deff]{Notation}

\newtheorem{ex}[deff]{Example}

\theoremstyle{plain}
\newtheorem{teo}[deff]{Theorem}

\newtheorem{lem}[deff]{Lemma}

\newtheorem{prop}[deff]{Proposition}

\newtheorem{cor}[deff]{Corollary}

\author[M. Calzi]{Mattia Calzi}

\address{Dipartimento di Matematica, Universit\`a degli Studi di
	Milano, Via C. Saldini 50, 20133 Milano, Italy}
\email{{\tt mattia.calzi@unimi.it}} 

\thanks{{\em Math Subject Classification 2020}: 32M15, 31C10 
}
\thanks{The author is a member of the 	Gruppo Nazionale per l'Analisi
	Matematica, la Probabilit\`a e le	loro Applicazioni (GNAMPA) of
	the Istituto Nazionale di Alta Matematica (INdAM). The author was partially funded by the INdAM-GNAMPA Project CUP\_E53C22001930001.
} 

\begin{abstract}
	Given a bounded symmetric domain $D$, we study (positive) pluriharmonic functions on $D$ and investigate a possible analogue of the family of Clark measures associated with a holomorphic function from $D$ into the unit disc in $\C$.  
\end{abstract}
\maketitle

\section{Introduction}

Let $\Db$ denote the unit disc in $\C$, and take a holomorphic function $\phi\colon \Db \to \Db$. Then, for every $\alpha$ in the torus $\T$,
\[
\Re\left( \frac{\alpha+ \phi}{\alpha-\phi} \right)=\frac{1-\abs{\phi}^2}{\abs{\alpha-\phi}^2}
\]
is a positive harmonic function. By the Riesz--Herglotz representation theorem, there is a unique positive Radon measure $\mi_\alpha$ on $\T$ such that
\[
\Re\left( \frac{\alpha+ \phi(z)}{\alpha-\phi(z)} \right)=(\Pc \mi_\alpha)(z)=\int_\T \frac{1-\abs{z}^2}{\abs{\alpha'-z}^2}\,\dd \mi_\alpha(\alpha')
\]
for every $z\in \Db$, where $\Pc$ denotes the Poisson integral operator (on $\T$). Equivalently,
\[
 \frac{\alpha+ \phi(z)}{\alpha-\phi(z)}=i \Im \frac{\alpha+ \phi(0)}{\alpha-\phi(0)} + \int_\T \frac{\alpha'+z}{\alpha'-z}\,\dd \mi_\alpha(\alpha')
\]
for every $z\in \Db$. The measures $\mi_\alpha$, $\alpha\in \T$ are called Clark (or Aleksandrov--Clark) measures, and were introduced by D.\ N.\ Clark in~\cite{Clark} in order to study the restricted shift operator. Their main properties were then established by A.\ B.\ Aleksandrov (cf., e.g.,~\cite{Aleksandrov1,Aleksandrov2,Aleksandrov3,Aleksandrov4,Aleksandrov5}).
Among these, $\mi_\alpha$ is singular (with respect to the normalized Haar measure $\beta_\T$ on $\T$) if and only if $\phi$ is inner, that is, $\phi$ has radial limits of modulus $1$ at almost every point in $\T$. In addition, 
\[
\int_\T \mi_\alpha\,\dd \beta_\T(\alpha)=\beta_\T
\]
in the sense of integrals of positive measures. This is known as Aleksandrov's disintegration theorem, since its first formulation was only concerned with inner functions, in which case the above integral of measures is indeed a disintegration of $\beta_\T$ relative to its pseudo-image measure $\beta_\T$ under $\phi$ (more precisely, under the boundary value function associated with $\phi$). Cf., e.g.,~\cite{CimaMathesonRoss,PoltoratskiSarason} for an account of the classical theory of Clark measures.

Clark measures proved to be a valuable tool in the study of both the theory of holomorphic functions in the unit disc and in the study of contractions.  It is then natural to investigate whether they may be extended to more general settings. If $U$ denotes the unit ball in $\C^n$, then two different possible analogues of Clark measures were introduced. On the one hand, inspired by the link with the study of contractions, M.\ T.\ Jury and then R.\ T.\ W.\ Martin (cf.~\cite{Jury1,Jury2,JuryMartin}) developed a theory of operator-valued Clark measures which may be used to study row contractions. These measures are in some sense linked to the choice of the Drury--Arveson space as the `correct' generalization of the Hardy space on $\Db$. On the other hand, more recently A.\ B.\ Aleksandrov and E.\ Doubtsov (cf.~\cite{AleksandrovDoubtsov,AleksandrovDoubtsov2}) introduced another possible extension of Clark measures associated with holomorphic functions $\phi\colon U\to \Db$, defined as the positive Radon measures $\mi_\alpha$ on $\partial U$ such that
\[
\Re\left( \frac{\alpha+ \phi(z)}{\alpha-\phi(z)} \right)=(\Pc \mi_\alpha)(z)=\int_{\partial U} \frac{(1-\abs{z}^2)^n}{\abs{1-\langle z\vert \zeta\rangle}^{2 n}}\,\dd \mi_\alpha(\zeta)
\]
for every $z\in U$, where $\Pc$ denotes the Poisson integral operator (on $\partial U$). 
It is important to notice that, whereas in the classical case \emph{every} non-zero positive Radon measure on $\T$ is the Clark measure associated with some holomorphic function $\phi\colon \Db\to \Db$, in this more general situation this is no longer the case. In fact, the Poisson integral $\Pc \mi_\alpha$ must be a positive \emph{pluriharmonic} function (and for this reason measures such as $\mi_\alpha$ are sometimes called `pluriharmonic'). This property dictates some severe restrictions on $\mi_\alpha$ (essentially since the Poisson kernel is far from being pluriharmonic in each variable).  For example, as shown in~\cite{AleksandrovDoubtsov}, if $\pi$ denotes the canonical projection of $\partial U$ onto the projective space $P^{n-1}$ and $\widehat \beta$ denotes the image of the normalized measure $\beta$ on $\partial U$ under $\pi$, then $\widehat\beta$ is a pseudo-image measure of   $\mi_\alpha$ under $\pi$ (that is, a subset $N$ of $P^{n-1}$ is $\widehat \beta$-negligible if and only if $\pi^{-1}(N)$ is $\mi_\alpha$-negligible), and $\mi_\alpha$ has a disintegration relative to $\widehat \beta$. In fact, it has a \emph{vaguely continuous} disintegration $(\mi_{\alpha,\xi})_{\xi\in P^{n-1}}$, and the $\mi_{\alpha,\xi}$ are nothing but the classical Clark measures (on $\pi^{-1}(\xi)$) associated with the restriction of $\phi$ to the disc $\Db \pi^{-1}(\xi)$ (cf.~Proposition~\ref{prop:12}). This disintegration theorem, in particular, shows that $\mi_\alpha$ is necessarily absolutely continuous with respect to the Hausdorff measure $\Hc^{n-2}$ (cf.~Corollary~\ref{cor:2}).

In this paper we shall generalize further this second approach to Clark measures to the case of bounded symmetric domains (of which $\Db$ and  $U$ are the simplest examples). In fact, we shall fix a circular (hence convex, cf.~\cite[Corollary 4.6]{Loos}) bounded symmetric domain $D\subseteq \C^n$, that is, a (circular convex) connected open subset of $\C^n$ such that for every $z\in D$ there is a (unique) holomorphic involution having $z$ as an isolated (actually, unique) fixed point. It is known that $D$ is then homogeneous (that is, that the group of its biholomorphism acts transitively on $D$), but converse is false (even though every  homogeneous bounded domain which has some `canonical' \emph{convex} realization is necessarily symmetric, cf.~\cite{Kai}). The reason why we chose to deal with bounded symmetric domains is twofolod: one the one hand, the Poisson integral operator on a bounded symmetric domain $D$ is well studied, and enjoys several nice properties which allow, for example, to ensure that the correspondence between inner functions and singular Clark measures be preserved; on the other hand, every bounded symmetric domain has a circular convex realization, and this allows to perform the disintegration trick introduced in~\cite{Aleksandrov} which allows to extend several results from positive measures on $\T$ to positive \emph{pluriharmonic} measures on the \v Silov boundary $\bD$ of $D$. 
Given a holomorphic function $\phi\colon D\to \Db$ and $\alpha\in \T$, we shall then define the Clark measure $\mi_\alpha$ (on $\bD$) so that
\[
\Re\left( \frac{\alpha+ \phi }{\alpha-\phi} \right)=\Pc \mi_\alpha
\]
on $D$, where $\Pc$ denotes the Poisson integral operator on $\bD$. Denote by $\beta$ the unique normalized positive Radon measure on $\bD$ which is invariant under all linear automorphisms of $D$. As for the case of the unit ball $U$, denoting by $\pi\colon \bD\to \widehat{\bD}$ the canonical projection of $\bD$ onto its quotient under the action $(\alpha,\zeta)\mapsto \alpha\zeta$ of the torus, the image $\widehat \beta$ of  $\beta$  under $\pi$ is a pseudo-image measure of $\mi_\alpha$ under $\pi$. In addition, $\mi_\alpha$ has a \emph{vaguely continuous} disintegration $(\mi_{\alpha,\xi})_{\xi\in \widehat{\bD}}$ relative to $\widehat \beta$, the $\mi_{\alpha,\xi}$ being nothing but the classical Clark measures (on $\pi^{-1}(\xi)$) associated with the restriction of $\phi$ to the disc $\Db \pi^{-1}(\xi)$. As before, this allows to show that $\mi_\alpha$ is absolutely continuous with respect to the Hasudorff measure $\Hc^{m-1}$, where $m$ denotes the (real) dimension of the (real) analytic manifold $\bD$. 

One may then prove the analogues of several properties of the classical Clark measures. We mention, though, that the results in~\cite{AleksandrovDoubtsov,AleksandrovDoubtsov2} are sharper, since in the case of the unit ball $U$ more tools are available, such as Henkin's and Cole--Range's theorem, whose validity on more general symmetric domains is unknown to us. As a matter of fact, both Henkin's and Cole--Range's theorem do \emph{not} hold in any \emph{reducible} bounded symmetric domain which admits non-trivial inner functions (cf.~Proposition~\ref{prop:17}).

Here is a plan of the paper. 
In Section~\ref{sec:2}, we review some basic properties of the Poisson integral operator on $\bD$ and prove some basic properties of pluriharmonic measures.
In Section~\ref{sec:3}, we prove some results on the disintegration of pluriharmonic measures, as well as some consequences.
In Section~\ref{sec:4}, we introduce and study the analogue of Clark measures described above. In particular, we shall consider the connection with possible analogues of the classical model and de Branges-Rovnyak spaces.
In Section~\ref{sec:5}, we summarize some implications between the theory of Henkin measures and Clark measures. 
In Section~\ref{sec:6}, we state the analogue of the characterization of the essential norm for composition operators; we provide no proofs since they amount to an almost word-by-word transcription of~\cite[Section 6]{AleksandrovDoubtsov}.
In Section~\ref{sec:7}, we briefly review the classical correspondence between Clark measures and angular derivatives, making use of the recent results proved in~\cite{MackeyMellon}.
Finally, in the appendix we provide some background information on the disintegration of (not necessarily positive) measures.

The author would like to thank M.\ M.\ Peloso for several discussions on the subject.

\section{Poisson Integrals}\label{sec:2}

\begin{nott}
	We denote by $D$ a bounded, circular,\footnote{A subset of $\C^n$ is called circular or circled if it is invariant under the action of the torus $\T\times \C^n\ni (\alpha,z)\mapsto \alpha z\in \C^n$.} and convex symmetric domain in a complex space $\C^n$, and by   $\bD$ its \v Silov boundary, that is, the smallest closed subset of $\partial D$ such that $\sup_D \abs{f}=\max_{\bD} \abs{f}$ for every $f\in \Hol(D)$ which extends by continuity to the closure $\Cl(D)$ of $D$. 
	We denote by $K$ the (compact) group of linear automorphisms of $D$, so that $K$ acts transitively on $\bD$ (cf., e.g.,~\cite[Corollary 4.9, Theorem 5.3, and Theorem 6.5]{Loos}). We assume that $K\subseteq U(n)$ and that $\bD\subseteq \partial B(0,1)$.
	
	We denote by $\Db$ the unit disc in $\C$, and by $\T$ its boundary, that is, the torus. We denote by $\beta_\T$ the normalized Haar measure on $\T$.	
\end{nott}

\begin{oss}
	Notice that, under the above assumptions, $D\subseteq B(0,1)$ (cf., e.g.,~\cite[Theorem 6.5]{Loos}). 
\end{oss}

\begin{deff}
	We denote by $\beta$ the unique $K$-invariant probability measure  on $\bD$.\footnote{Observe that the Hausdorff measure $\Hc^m$ (on $\bD$, relative to the Euclidean distance), where $m$ is the (real) dimension of the real analytic manifold $\bD$, is $K$-invariant since $K\subseteq U(n)$. Consequently, $\beta=\frac{1}{\Hc^m(\bD)}\Hc^m$.}
	In addition, we denote by $\widehat{\bD}$ the quotient of $\bD$ by the canonical action of the torus $\T$,   by $\pi\colon \bD\to \widehat{\bD}$ the canonical projection, and by $\widehat\beta\coloneqq \pi_*(\beta)$ the image of $\beta$ under $\pi$.
	
	For every $\xi\in \widehat {\bD}$, define
	\[
	\Db_\xi\coloneqq\Set{w\zeta\colon w\in \C, \abs{w}<1, \zeta\in \pi^{-1}(\xi)}=\Db \pi^{-1}(\xi).
	\] 
	We denote by $\beta_\xi$ the   $\T$-invariant normalized measured on $\pi^{-1}(\xi)$, that is, the image of $\beta_\T$ under the mapping $\T\ni \alpha \mapsto \alpha \zeta \in \bD$, for every $\zeta\in \pi^{-1}(\xi)$. We shall also interpret $\beta_\xi$ as a Radon measure on $\bD$, supported in $\pi^{-1}(\xi)$.
\end{deff}

\begin{nott}
	As a general rule, we shall use the letters $z$ and $\zeta$ to denote elements of $D$ and $\bD$, respectively, $\xi$ to denote an element of $\widehat{\bD}$, $\rho$ to denote an element of $[0,1)$, and $w,\alpha$ to denote elements of $\Db$ and $\T$, respectively. We hope that these conventions may help the reader navigate through the manuscript.
\end{nott}

\begin{oss}\label{oss:13}
	Observe that $(\beta_\xi)$ is a disintegration of $\beta$ under $\pi$.
\end{oss}

\begin{proof}
 	Observe that the mapping $\xi \mapsto \beta_\xi$ is vaguely continuous, so that 
	\[
	\beta'\coloneqq \int_{\widehat {\bD}} \beta_\xi\,\dd \widehat \beta(\xi)=\int_{\bD} \beta_{\pi(\zeta)}\,\dd \beta(\zeta)
	\]
	is a well defined probability Radon measure. Now, for every $k\in K$,
	\[
	k_*\beta'= \int_{\bD} k_*\beta_{\pi(\zeta)}\,\dd \beta(\zeta)= \int_{\bD} \beta_{\pi(k \zeta)}\,\dd \beta(\zeta)=\beta'
	\] 
	since $\beta$ is $K$-invariant. Thus, $\beta'$ is $K$-invariant. Since $K$ acts transitively on $\bD$, and since both $\beta$ and $\beta'$ are probabily measures, this shows that $\beta=\beta'$.
\end{proof}
 
\begin{nott}
	We denote by $\Pc_\R$ the space of polynomial mappings from the real space underlying $\C^n$  into $\C$. In addition, we denote by $\Pc_\C$ and $\overline{\Pc_\C}$ the spaces of holomorphic and anti-holomorphic polynomials on $\C^n$, respectively.
	
	In addition, we shall denote by $\Mc^1(\bD)$ the space of Radon measures on $\bD$, endowed with the norm $\mi\mapsto \abs{\mi}(\bD)$. 
\end{nott}

\begin{oss}\label{oss:1}
	The canonical image of $\Pc_\R$ is dense in $C(\bD)$ by the Stone--Weierstrass theorem, hence also in $L^2(\beta)$.
\end{oss}

Recall that a reproducing kernel Hilbert space of holomorphic functions on $D$ is a Hilbert space $H$ which embeds continuously into the space of holomorphic functions $\Hol(D)$ on $D$ (endowed with the topology of compact convergence). In this case, point evaluations are continuous, so that there is a function $k\colon D\times D\to \C$ such that $k(\,\cdot\,,z)\in H$ and $f(z)= \langle f \vert k(\,\cdot\,,z)\rangle_H$ for every $z\in D$ and for every $f\in H$. It turns out that $k(z,z')=\overline{k(z',z)}$ for every $z,z'\in D$ and that the set of $k(\,\cdot\,,z)$, as $z$ runs through $D$, is total (that is, generates a dense vector subspace) in $H$. Furthermore, if $(e_j)$ is an orthonormal basis of $H$, then $k=\sum_j e_j\otimes \overline{e_j}$, with locally uniform convergence on $D\times D$.

\begin{deff}
	We denote by $\Cc$ the Cauchy--Szeg\H o kernel of $D$, that is, the reproducing kernel of the Hardy space 
	\[
	H^2(D)=\Set{f\in \Hol(D)\colon \sup_{0\meg \rho<1} \int_{\bD} \abs{f(\rho \zeta)}^2\,\dd \beta(\zeta)<\infty },
	\]
	and define the Poisson--Szeg\H o kernel as\footnote{As we shall recall in Proposition~\ref{prop:1}, $\Cc$ extends to a sesquiholomorphic function on a neighbourhood of $D\times \Cl(D)$, so that we may safely consider $\Cc$ also on $D\times \bD$.}
	\[
	\Pc\colon D\times \bD\ni (z,\zeta)\mapsto \frac{\abs{\Cc(z,\zeta)}^2}{\Cc(z,z)}\in \R_+.
	\] 
	Define, for every (Radon) measure $\mi$ on $\bD$,
	\[
	\Cc(\mi)\colon D\ni z\mapsto \int_{\bD} \Cc(z,\zeta)\,\dd \mi(\zeta)\in \C,
	\]
	\[
	\Pc(\mi)\colon D\ni z\mapsto \int_{\bD} \Pc(z,\zeta)\,\dd \mi(\zeta)\in \C
	\]
	and
	\[
	\Hc(\mi)\colon D\ni z\mapsto \int_{\bD} (2\Cc(z,\zeta)-1)\,\dd \mi(\zeta)\in \C.
	\]
\end{deff}

If $f\in L^1(\beta)$, we shall also write $\Cc(f)$, $\Pc(f)$, and $\Hc(f)$ instead of $\Cc(f\cdot \beta)$, $\Pc(f\cdot \beta)$, and $\Hc(f\cdot \beta)$, respectively. In order to simplify the notation, we shall   also simply write $\Cc \mi$, $\Pc \mi$, and $\Hc \mi$ instead of $\Cc(\mi)$, $\Pc(\mi)$, and $\Hc(\mi)$, respectively.

\begin{oss}
	When $D=\Db$, one has
	\[
	\Cc(w,w')=\frac{1}{1-w \overline{w'}}
	\]
	for every $w,w'\in \Db$, so that
	\[
	\Pc(w,\alpha)=\frac{1-\abs{w}^2}{\abs{\alpha-w}^2}
	\]
	and
	\[
	2\Cc(w,\alpha)-1=\frac{\alpha+w}{\alpha-w}
	\]
	for every $w\in \Db$ and for every $\alpha\in \T$. In this case, $\Pc \mi=\Re \Hc\mi$ for every \emph{real} measure $\mi$ on $\T$.
\end{oss}

\begin{lem}\label{lem:1}
	The following hold:
	\begin{enumerate}
		\item[\textnormal{(1)}]  if $P,Q\in \Pc_\C$ are homogeneous  with different degrees, then $\langle P\vert Q\rangle_{L^2(\beta)}=\langle \overline P\vert \overline Q\rangle_{L^2(\beta)}=0$;
		
		\item[\textnormal{(2)}]  if $P,Q\in \Pc_\C$, then $\langle P\vert \overline Q\rangle_{L^2(\beta)}=P(0)Q(0)$. In particular, $\Cc(0,\,\cdot\,)=\Pc(0,\,\cdot\,)=1$ on $\bD$.
	\end{enumerate}
\end{lem}

\begin{proof}
	(1) Let $h$ and $k$ be the degrees of $P$ and $Q$. Then, for every $\theta\in \T$,
	\[
	\langle P\vert Q\rangle_{L^2(\beta)}=\langle P(\theta\,\cdot\,)\vert Q(\theta\,\cdot\,)\rangle_{L^2(\beta)}=\theta^{h-k}\langle P\vert Q\rangle_{L^2(\beta)},
	\]
	thanks to the $K$-invariance of $\beta$ and to the circularity of $D$. Therefore, $\langle P\vert Q\rangle_{L^2(\beta)}=0$.
	
	(2) It suffices to observe that, by (1) and 
	\[
	(PQ)(0)=\langle PQ \vert 1\rangle_{L^2(\beta)}=\langle P\vert \overline Q\rangle_{L^2(\beta)}= \Cc(PQ)(0)
	\]
	since $PQ\in H^2(D)$, so that $\Cc(0,\,\cdot\,)=1$ by the arbitrariness of $P,Q$.
\end{proof}

\begin{deff}
	Define $\Cc^{(k)}$, for every $k\in\N$, as the reproducing kernel of the space of homogeneous holomorphic polynomials of degree $k$,  endowed with the scalar product induced by $H^2(D)$. We define $\Cc^{(k)}(\mi)$, for every Radon measure $\mi$ on $\bD$, accordingly.
\end{deff}

\begin{oss}\label{oss:16}
	Notice that
	\[
	\Cc(z,z')=\sum_{k\in\N} \Cc^{(k)}(z,z')
	\]
	for every $z,z'\in D$, with locally uniform convergence on $D\times \Cl(D)$. When $D$ is irreducible, this follows from~\cite[Theorem 3.8 and the following Remark 1]{FarautKoranyi}; the general case may be deduced from the irreducible one. 
\end{oss}

\begin{deff}
	Given a function $f$ on $D$, we define 
	\[
	f_\rho\colon \bD\ni \zeta\mapsto f(\rho \zeta)\in \C
	\]
	for every $\rho\in [0,1)$. In addition, we define $f_1$ as the pointwise limit (where it exists) of the $f_\rho$ for $\rho\to 1^-$.
\end{deff}

\begin{oss}
	Take $f\in C(D)$. Then, the domain of $f_1$ is a Borel subset of $\bD$ and $f_1$ is Borel measurable thereon.
\end{oss}

\begin{proof}
	The domain $B$ of $f_1$ is the set 
	\[
	\bigcap_{j\in\N} \bigcup_{h\in \N}   \Set{ \zeta\in \bD\colon  \forall \rho, \rho'\in (1-2^{-h},1) \:\: \abs{f_\rho(\zeta)-f_{\rho'}(\zeta)}\meg 2^{-j}  }.
	\]
	Since the set $\Set{ \zeta\in \bD\colon  \forall \rho, \rho'\in (1-2^{-h},1) \:\: \abs{f_\rho(\zeta)-f_{\rho'}(\zeta)}\meg 2^{-j}  }$ is closed for every $j,h\in \N$, this proves that $B$ is a Borel subset of $\bD$. Then, $f_1$ is the pointwise limit of the sequence of the restrictions of the $f_{1-2^{-h}}$ to $B$ (for $h\to \infty$), hence is Borel measurable.
\end{proof}

\begin{prop}\label{prop:1}
	The following hold:
	\begin{enumerate}
		\item[\textnormal{(1)}]  $\Cc$ extends to a sesqui-holomorphic function on a neighbourhood of $D\times \Cl(D)$, where $\Cl(D)$ denotes the closure of $D$ in $\C^n$;
		
		\item[\textnormal{(2)}]  $\Pc$ is real analytic and  nowhere vanishing on $D\times \bD$;
		
		\item[\textnormal{(3)}]  $\Pc(z,\,\cdot\,)\cdot \beta$ is a probability measure on $\bD$ for every $z\in D$;
		 
		\item[\textnormal{(4)}]  for every $f\in C(\bD)$, 
		\[
		\lim_{\rho\to 1^-} (\Pc f)_\rho =f
		\]
		uniformly on $\bD$;
		
		\item[\textnormal{(5)}]  for every Radon measure $\mi$ on $\bD$,
		\[
		\lim_{\rho\to 1^-} (\Pc \mi)_\rho \cdot \beta=\mi
		\]
		in the vague topology, and $\lim\limits_{\rho \to 1^-} \norm{(\Pc \mi)_\rho}_{L^1(\beta)}=\norm{\mi}_{\Mc^1(\bD)}$;
		
		\item[\textnormal{(6)}]  for every Radon measure $\mi$ on $\bD$, $(\Pc \mi)_1$ is defined almost everywhere and $(\Pc\mi)_1\cdot \beta$  is the absolutely continuous part of $\mi$ with respect to $\beta$.
	\end{enumerate}
\end{prop}

Concerning (6), notice that~\cite[Theorem 3.6]{SteinWeiss}  actually shows (strictly speaking, in the context of symmetric Siegel domains) that   $(\Pc \mi)(z)$   converges  to $(\Pc \mi)_1(\zeta)$, as $z\to \zeta$ `restrictedly and admissibly,' for  $\beta$-almost every $\zeta\in \bD$. Since radial convergence is sufficient for our purposes, we shall \emph{not} recall what the phrase `restrictedly and admissibly' means in this context.

\begin{proof}
	Observe that, when $D$ is irreducible,~\cite[Theorem 3.8 and the following Remark 1]{FarautKoranyi} shows that there is a sesqui-holomorphic polynomial $P\colon \C^n\times \C^n\to \C$ such that $\Cc^2=P^{-1}$ on $D\times D$, and such that $P$ vanishes nowhere  on $D\times \Cl(D)$ (here we are using the fact that $2 n/r$ is an integer if $D$ is irreducible, where $r$ denotes the rank of $D$). This proves (1) and (2) in this case. The general case then follows.
	Then, (3) to (5) are consequences of~\cite[Theorem 4.7 and Remark 4.9]{Koranyi}. 
	
	Finally, (6) follows from~\cite[Theorem 3.6]{SteinWeiss} (which is stated for symmetric Siegel domains, but can be transferred to bounded symmetric domains).  
\end{proof}

\begin{lem}\label{lem:3}
	For every $z,z'\in D$ and for every $P\in \Pc_\C+\overline{\Pc_\C}$,
	\[
	\int_{\bD} (\Cc(z,\zeta)-\Cc(z,z'))(\overline{\Cc(z',\zeta)}-\Cc(z,z')) P(\zeta)\,\dd \beta(\zeta)=(\Cc(z,z')^2 -\Cc(z,z') )P(0).
	\]
	In addition, for every $h,k\in \N$ and for every $P\in \Pc_\C+\overline{\Pc_\C}$,
	\[
	\int_{\bD} \Cc^{(k)}(z,\zeta) \overline{\Cc^{(h)}(z',\zeta)} P(\zeta)\,\dd \beta(\zeta)=\Cc^{\min(h,k)}(z,z') \begin{cases}
		\int_{\bD} \Cc^{(k-h)}(z,\zeta)P(\zeta)\,\dd \beta(\zeta) &\text{if $k>h$}\\
		\int_{\bD} P(\zeta)\,\dd \beta(\zeta)=P(0) & \text{if $k=h$}\\
		\int_{\bD} \overline{\Cc^{(h-k)}(z',\zeta)} P(\zeta)\,\dd \beta(\zeta) &\text{if $k<h$}.
	\end{cases}
	\]
\end{lem}

\begin{proof}
	Assume first that $P\in \Pc_\C$. Then
	\[
	\int_{\bD} \Cc(z,\zeta)(\overline{\Cc(z',\zeta)}-\Cc(z,z')) P(\zeta)\,\dd \beta(\zeta)=\Cc(  (\overline{\Cc(z',\,\cdot\,)}-\Cc(z,z')) P)(z)=( \overline{\Cc(z',z)}-\Cc(z,z')) P(z)=0
	\]
	since $\overline{\Cc(z',\,\cdot\,)} P=\Cc(\,\cdot\,,z') P\in H^2(D)$. Analogously,
	\[
	\int_{\bD} (\overline{\Cc(z',\zeta)}-\Cc(z,z')) P(\zeta)\,\dd \beta(\zeta)=\Cc( (\overline{\Cc(z',\,\cdot\,)}-\Cc(z,z')) P)(0)=(\overline{\Cc(z',0)}-\Cc(z,z')) P(0),
	\]
	whence the result when $P\in \Pc_\C$, since $\Cc(z',0)=1$. When $P\in \overline{\Pc_\C}$, applying a conjugate leads to the same conclusion, whence the first assertion.

	Now, take $P,Q\in \Pc_\C$ and $h,k\in \N$, and let us prove that
	\[
	\int_{\bD} \Cc^{(k)}(z,\zeta) \overline{\Cc^{(h)}(z',\zeta)} (P(\zeta)+\overline{Q(\zeta)})\,\dd \beta(\zeta)= \overline{\Cc^{(h)}(z',z)} \Cc^{(k-h)}(P)(z)+\overline{ \overline{\Cc^{(k)}(z,z') } \Cc^{(h-k)}(Q)(z')},
	\]
	where one must interpret $\Cc^{(\ell)}(P)=\Cc^{(\ell)}(Q)=0$ if $\ell<0$.
	Indeed,  Lemma~\ref{lem:1} shows that
	\[
	\begin{split}
	\int_{\bD} \Cc^{(k)}(z,\zeta) \overline{\Cc^{(h)}(z',\zeta)} P(\zeta)\,\dd \beta(\zeta)&=\sum_\ell \langle \Cc^{(h)}(\,\cdot\,,z') \Cc^{(\ell)}(P)\vert \Cc^{(k)}(\,\cdot\,,z)\rangle_{L^2(\beta)}\\
		&= \langle \Cc^{(h)}(\,\cdot\,,z') \Cc^{(k-h)}(P)\vert \Cc^{(k)}(\,\cdot\,,z)\rangle_{L^2(\beta)}\\
		& 		=\Cc^{(h)}(z,z') \Cc^{(k-h)}(P)(z).
	\end{split}
	\]
	In a similar way (or taking conjugates), one shows that 
	\[
	\int_{\bD} \Cc^{(k)}(z,\zeta) \overline{\Cc^{(h)}(z',\zeta) Q(\zeta)}\,\dd \beta(\zeta)=\overline{ \overline{\Cc^{(k)}(z,z') } \Cc^{(h-k)}(Q)(z')}.
	\]
	Now, clearly $\Cc^{(0)}=1$, while
	\[
	\Cc^{(k-h)}(\overline Q)=0
	\]
	when $k>h$, and
	\[
	\Cc^{(h-k)}(\overline P)=0
	\]
	when $k<h$, by Lemma~\ref{lem:1}, so that	 the second assertion follows.
\end{proof}

\begin{oss}\label{oss:15}
	Take $f\in \Hol(D)$.  
	Then, the homogeneous expansion of $f$ converges (to $f$) locally uniformly on $D$.
\end{oss}

\begin{proof}
	Let $\sum_j P_j$ be the homogeneous expansion of $f$, so that $P_j=\Cc^{(j)} (f)$ for every $j\in\N$. Take $\rho\in (0,1)$. Since $f_\rho\in C(\bD)$, by means of Remark~\ref{oss:16} we see that $\sum_j \rho^j P_j$ converges to $f(\rho\,\cdot\,)$ locally uniformly on $D$. The assertion follows by the arbitrariness of $\rho$. 
\end{proof}
 
\begin{deff}
	Let $V$ be a subset of $\Pc_\R$. Then, we define $\Pc_\R\cap V^\perp\coloneqq \Set{P\in \Pc_\R\colon \forall Q\in V\:\: \langle P\vert Q\rangle_{L^2(\beta)}=0}$. 
\end{deff}

\begin{prop}\label{prop:4}
	Let $\mi$ be a \emph{real} measure on $\bD$. Then, the following conditions are equivalent:
	\begin{enumerate}
		\item[\textnormal{(1)}] $\Pc(\mi)$ is pluriharmonic;
		
		\item[\textnormal{(2)}]  $\Pc(\mi)=\Re \Hc(\mi)$;
		
		\item[\textnormal{(3)}]  $\langle \mi, P\rangle=0$ for every $P\in \Pc_\R\cap (\Pc_\C\cup\overline{\Pc_\C})^\perp$;
		
		\item[\textnormal{(4)}]  $\mi$ is in the vague closure of $\Re \Pc_\C \cdot \beta$.
	\end{enumerate}
\end{prop}

Recall that a function is pluriharmonic if its restriction to every complex line is harmonic. In addition, a real-valued function is pluriharmonic if and only if it may be written locally as the real part of a holomorphic function (cf., e.g.,~\cite[Proposition 2.2.3]{Krantz}). Since $D$ is convex, every real-valued pluriharmonic function on $D$ is (globally) the real part of a holomorphic function on $D$.

\begin{proof}
	(1) $\implies$ (4). Observe that $(\Pc\mi)_\rho\cdot \beta$ converges vaguely  to $\mi$ as $\rho\to 1^-$ by Proposition~\ref{prop:1}, so that it will suffice to prove our assertion for $(\Pc\mi)_\rho\cdot \beta$, for every $\rho\in (0,1)$. Since $\Pc(\mi)$ is pluriharmonic and $D$ is convex, there is a holomorphic function $f$ on $D$ such that $\Pc(\mi)=\Re f$. In addition,  the homogeneous expansion $\sum_j P_j$ of $f$ converges locally uniformly to $f$ on $D$ (cf.~Remark~\ref{oss:15}). In particular, $(\Pc\mi)_\rho$ is the uniform  limit of the sequence $\sum_{j\meg N} \Re ( \rho^{j} P_j)$ of elements of $\Re \Pc_\C$. 
	
	(4) $\implies$ (3). It will suffice to prove that $\langle P\vert Q\rangle_{L^2(\beta)}=0$ for every $P\in \Re \Pc_\C$ and for every $Q\in \Pc_\R\cap (\Pc_\C\cup\overline{\Pc_\C})^\perp=\Pc_\R\cap (\Pc_\C+\overline{\Pc_\C})^\perp$. Since $\Re \Pc_\C\subseteq\Pc_\C+ \overline{\Pc_\C}$, the assertion follows.
	
	(3) $\implies$ (2). Observe that 
	\[
	\Re (\Hc \mi)(z)= \int_{\bD} (\Cc(z,\zeta)+ \overline{\Cc(z,\zeta)} -1  )\,\dd \mi(\zeta),
	\]
	and that
	\[
	(\Pc \mi)(z)= \frac{1}{\Cc(z,z)}\int_{\bD} \Cc(z,\zeta)\overline{\Cc(z,\zeta)}\,\dd \mi(\zeta)
	\]
	for every $z\in D$.
	In addition, Remark~\ref{oss:16} implies that
	\[
	\Cc(z,\,\cdot\,)=\sum_{k\in \N} \Cc^{(k)}(z,\,\cdot\,)
	\]
	uniformly on $\bD$, for every $z\in D$. 
	Therefore, by means of Lemma~\ref{lem:3} we see that
	\[
	\begin{split}
		\int_{\bD} \Cc(z,\zeta)\overline{\Cc(z,\zeta)}\,\dd \mi(\zeta)&=\sum_{h,k\in\N} \int_{\bD} \Cc^{(k)}(z,\zeta)\overline{\Cc^{(h)}(z,\zeta)}\,\dd \mi(\zeta)\\
			&=\sum_{k\in\N}\Cc^{(k)}(z,z) \int_{\bD}  \Big(1+\sum_{h>0} \Cc^{(h)}(z,\zeta)+\sum_{h>0} \overline{\Cc^{(h)}(z,\zeta)}\Big) \,\dd \mi(\zeta)\\
			&=	\Cc(z,z) \int_{\bD} (1+\Cc(z,\zeta)-1+\overline{\Cc(z,\zeta)}-1 )\,\dd \mi(\zeta)\\
			&=\Cc(z,z)\Re (\Hc \mi)(z)
	\end{split}
	\]
	for every $z\in D$, whence (2).
	
	(2) $\implies$ (1). Obvious, since $\Hc(\mi)$ is holomorphic.
\end{proof}

\begin{prop}\label{prop:11}
	Let $\mi$ be a complex measure on $\bD$. Then, the following conditions are equivalent:
	\begin{enumerate}
		\item[\textnormal{(1)}]  $\Pc(\mi)$ is holomorphic;
		
		\item[\textnormal{(2)}]  $\Pc(\mi)=\Cc(\mi)$;
		
		\item[\textnormal{(3)}]  $\langle \mi, \overline P\rangle=0$ for every $P\in \Pc_\R\cap \Pc_\C^\perp$;
		
		\item[\textnormal{(4)}]  $\mi$ is in the vague closure of $\Pc_\C \cdot \beta$.
	\end{enumerate} 
\end{prop}

\begin{proof}
	(1) $\implies$ (4). Observe that $(\Pc\mi)_\rho\cdot \beta$ converges vaguely  to $\mi$ as $\rho\to 1^-$ by Proposition~\ref{prop:1}, so that it will suffice to prove our assertion for $(\Pc\mi)_\rho\cdot \beta$, for every $\rho\in (0,1)$. Now, the homogeneous expansion $\sum_j P_j$ of $\Pc(\mi)$ converges locally uniformly to $\Pc(\mi)$ on $D$ (cf.~Remark~\ref{oss:15}). In particular, $(\Pc\mi)_\rho$ is the uniform   limit of the sequence $\sum_{j\meg N} \rho^{j} P_j$ of elements of $ \Pc_\C$.
	
	(4) $\implies$ (3). It will suffice to observe that $\langle P\vert Q\rangle_{L^2(\beta)}=0$ for every $P\in\Pc_\C$ and for every $Q\in \Pc_\R\cap \Pc_\C^\perp$. 
	
	(3) $\implies$ (2). Observe that 
	\[
	\Pc (\mi)(z)= \frac{1}{\Cc(z,z)}\int_{\bD} \Cc(z,\zeta)\overline{\Cc(z,\zeta)}\,\dd \mi(\zeta)
	\]
	for every $z\in D$.
	In addition, Remark~\ref{oss:16} implies that
	\[
	\Cc(z,\,\cdot\,)=\sum_{k\in \N} \Cc^{(k)}(z,\,\cdot\,)
	\]
	uniformly on $\bD$, for every $z\in D$. 
	In addition, it is clear that, for every $P\in \Pc_\C$,
	\[
	\int_{\bD} \Cc^{(k)}(z,\zeta)\overline{\Cc^{(h)}(z,\zeta)} P(\zeta)\,\dd \beta(\zeta)=\Cc^{(h)}(z,z) (\Cc^{(k-h)} P)(z),
	\]
	where $\Cc^{(k-h)} (P)=0$ when $k<h$.
	Therefore, 
	\[
	\begin{split}
		\int_{\bD} \Cc(z,\zeta)\overline{\Cc(z,\zeta)}\,\dd \mi(\zeta)&=\sum_{h,k\in\N} \int_{\bD} \Cc^{(k)}(z,\zeta)\overline{\Cc^{(h)}(z,\zeta)}\,\dd \mi(\zeta)\\
		&=\sum_{k\in\N}\Cc^{(k)}(z,z) \int_{\bD}  \sum_{h\in \N} \Cc^{(h)}(z,\zeta) \,\dd \mi(\zeta)\\
		&=	\Cc(z,z) \int_{\bD} \Cc(z,\zeta)\,\dd \mi(\zeta)\\
		&=\Cc(z,z)(\Cc\mi)(z)
	\end{split}
	\]
	for every $z\in D$, whence (2).	
	
	(2) $\implies$ (1). Obvious, since $\Cc(\mi)$ is holomorphic. 
\end{proof}

\begin{deff} 
	We say that a complex measure $\mi$ on $\bD$ is pluriharmonic or holomorphic  
	if $\Pc(\mi)$ is pluriharmonic or  holomorphic,   anti-holomorphic, 
	respectively.
\end{deff}

\begin{oss}\label{oss:4}
	Take $f\in C(\bD)$. Then, the following  conditions are equivalent:
	\begin{enumerate}
		\item[\textnormal{(1)}] $f\cdot \beta$ is holomorphic;
		
		\item[\textnormal{(2)}] there is $g\in \Hol(D)$ such that $g_\rho\to f$ uniformly on $\bD$ for $\rho\to 1^-$;
		
		\item[\textnormal{(3)}] $f$ belongs to the closed vector space generated by $\Pc_\C$ in $C(\bD)$;
		
		\item[\textnormal{(4)}] $f$ belongs to the closed vector space generated by the $\Cc(\,\cdot\,,z)$, $z\in D$, in $C(\bD)$.
	\end{enumerate}
\end{oss}

\begin{proof}
	(1) $\implies$ (2). By Propositions~\ref{prop:1} and~\ref{prop:11}, it suffices to take $g=\Pc f$.
	
	(2) $\implies$ (3). It suffices to prove the assertion for the $g_\rho$, $\rho\in (0,1)$. Then, $g_\rho$ is the uniform limit of (the restriction to $\bD$ of) the homogeneous expansion of $g$, by Remark~\ref{oss:15}, whence the assertion.
	
	(3) $\implies$ (4). Let $V$ be the closed vector subspace of $C(\bD)$ generated by the $\Cc(\,\cdot\,,z)$, $z\in D$, and observe that $V$ is $\T$-invariant. Since the action of $\T$ in $C(\bD)$ is continuous, we see that $V$ contains 
	\[
	\Cc^{(k)}(\,\cdot\,,z)=\int_{\T} \alpha^{-k}\Cc(\alpha \,\cdot\,, z)\,\dd \beta_\T(\alpha)
	\]
	for every $z\in D$ and for every $k\in \N$ (cf.~Lemma~\ref{lem:1}). By the reproducing properties of the $\Cc^{(k)}$, this proves that $V$ contains $\Pc_\C$, whence the assertion.
	
	(4) $\implies$ (1). It suffices to prove that $\Cc(\,\cdot\,,z)\cdot \beta$ is in the vague closure of $\Pc_\C\cdot \beta$ for every $z\in D$, by Proposition~\ref{prop:11}. It suffices to take the homogeneous expansion of $\Cc(\,\cdot\,,z)$, thanks to Remark~\ref{oss:16}.
\end{proof}

\begin{prop}\label{prop:5}
	The mapping $\mi \mapsto \Pc(\mi)$  induces a bijection of the space of positive pluriharmonic measures on $\bD$ onto the space of positive pluriharmonic functions on $D$.
	
	Equivalently, the mapping $\mi\mapsto \Hc(\mi)$ induces a bijection of the space of positive pluriharmonic measures on $\bD$ onto the space of holomorphic functions on $D$ with positive real part and which are real at $0$.
\end{prop}

Cf.~\cite[Theorem 1]{KoranyiPukanszki}, where the same theorem is stated for some more specific domains, but extension to this case (or more general ones) is indicated as possible with essentially the same techniques.

\begin{proof}
	The fact that the mapping $\mi\mapsto \Pc(\mi)$ is one-to-one follows from Proposition~\ref{prop:1}. In addition, the fact that the second assertion follows from the first one is a consequence of  Proposition~\ref{prop:4}. 
	In order to conclude, it will then suffice to prove that, if $f$ is a positive pluriharmonic function on $D$, then there is a positive (necessarily pluriharmonic) measure $\mi$ on $D$ such that $f=\Pc(\mi)$. Let $g$ be a holomorphic function such that $\Re g=f$.
	Take $\rho\in (0,1)$ and observe that $\Cc g_\rho= g(\rho\,\cdot\,)$, so that $\Pc f_\rho= \Re \Pc g_\rho=\Re g(\rho\,\cdot\,)=f(\rho\,\cdot\,)$ by Proposition~\ref{prop:11}. Since $\Pc(0,\cdot\,)=1$, this proves that $f_\rho\cdot \beta$ has norm equal to $f(0)$. 
	Therefore, given an ultrafilter $\Uf$ on $(0,1)$ which is finer than the filter `$\rho\to 1^-$', the $f_\rho\cdot \beta$ converge vaguely to some positive Radon  measure $\mi$ as $\rho$ runs along $\Uf$. Since $\Pc(f_\rho)(z)=f(\rho z)$ for every $\rho\in (0,1)$ and for every $z\in D$, the assertion follows passing to the limit.
\end{proof}

\begin{deff}
	By an inner function on $D$ we mean a bounded holomorphic function $\phi\colon D\to \C$ with boundary values of modulus $1$ almost everywhere on $\bD$ (that is, $\phi_1(\zeta)\in \T$ for $\beta$-almost every $\zeta\in \bD$).
\end{deff}

Notice that $\phi\in H^2(D)$, so that $\phi=\Pc \phi_1=\Cc \phi_1$. In particular, $\norm{\phi}_{L^\infty(D)}= 1$. 

\begin{prop}
	The mapping $\mi\mapsto \ee^{-\Hc(\mi)}$ induces a bijection of the set of positive pluriharmonic measures on $\bD$ which are singular with respect to $\beta$, and the set of nowhere-vanishing inner functions on $D$ which are $>0$ at $0$.
\end{prop}

\begin{proof}
	If $\mi$ is a positive pluriharmonic measure on $\bD$ which is singular with respect to $\beta$, then $\Re \Hc(\mi)\Meg 0$ on $D$ and $ \Re (\Hc\mi)_1=(\Pc \mi)_1=0$ $\beta$-almost everywhere, thanks to Propositions~\ref{prop:1} and~\ref{prop:4}. Therefore,
	\[
	\abs*{\ee^{-\Hc (\mi)}}\meg 1
	\]
	on $D$ and
	\[
	\abs*{\ee^{-(\Hc \mi)_1}}= 1
	\]
	$\beta$-almost everywhere. It is clear that $\ee^{-(\Hc\mi)(0)}=\ee^{-\norm{\mi}_{\Mc^1(\bD)}}>0$. 
	
	Conversely, let $\phi$ be a  nowhere-vanishing inner function on $\bD$ which is $>0$ at $0$. Since $D$ is convex, $-\log \phi$ may be defined uniquely as a holomorphic function on $D$ which is real at $0$. Then,
	\[
	\Re(-\log \phi)=-\log(\abs{\phi})\Meg 0,
	\]
	since $\norm{\phi}_{L^\infty(D)}=1$, so that Proposition~\ref{prop:5} shows that there is a positive pluriharmonic measure $\mi$ on $\bD$ such that $\Hc(\mi)=-\log\phi$. Since
	\[
	  \Re(-\log(\phi_1))=-  \log\abs{\phi_1}=0
	\]
	$\beta$-almost everywhere, Propositions~\ref{prop:1} and~\ref{prop:4} show that $\mi$ is singular with respect to $\beta$.
\end{proof}

\section{Disintegration of Pluriharmonic Measures}\label{sec:3}

\begin{teo}\label{prop:8}
	Let $\mi$ be a non-zero pluriharmonic measure on $\bD$. Then, $\widehat \beta$ is a pseudo-image measure of $\mi$ under $\pi$. 
	
	In addition, if we denote with   $(\mi_\xi)$ a disintegration of $\mi$ relative to $\widehat \beta$, then the following hold:
	\begin{itemize}
		\item for $\widehat \beta$-almost every $\xi\in \widehat{ \bD}$ and for every $z\in \Db_\xi$,
		\[
		(\Pc \mi)(z)= \int_{\bD} \frac{1-\abs{z}^2}{\abs{\zeta-z}^2}\,\dd \mi_\xi(\zeta);
		\] 
		
		\item if $\mi$ is holomorphic, then $\mi_\xi$ is holomorphic (in particular, absolutely continuous with respect to $\beta_\xi$) for almost every $\xi\in \widehat{ \bD}$.
	\end{itemize} 
\end{teo}

The proof extends~\cite[Proposition 2.1]{AleksandrovDoubtsov}.
Notice that we assume that $\mi$ be non-zero only to ensure that $\widehat\beta$ be a pseudo-image measure of $\mi$ under $\pi$. 

\begin{proof}
	Observe that $\Pc( \mi)$ is a pluriharmonic function on $D$. Consequently, the restriction $(\Pc \mi)^{(\xi)} $ of $\Pc(\mi)$ to $ \Db_\xi$ is  harmonic  for every $\xi\in \widehat{\bD}$. In addition,
	\[
	\begin{split}
	\int_{\widehat{\bD}} \sup_{0<\rho<1} \norm{ (\Pc \mi)^{(\xi)}_\rho }_{L^1(\beta_\xi)}\,\dd \widehat\beta(\xi)&=\int_{\widehat{\bD}} \lim_{\rho \to 1^-} \int_{\bD} \abs{(\Pc \mi)(\rho \zeta)}\,\dd \beta_\xi(\zeta)\,\dd \widehat\beta(\xi)\\
		&=\lim_{\rho \to 1^-}\int_{\widehat{\bD}}   \int_{\bD} \abs{(\Pc \mi)(\rho \zeta)}\,\dd \beta_\xi(\zeta)\,\dd \widehat\beta(\xi)\\
		&=\lim_{\rho \to 1^-} \int_{\bD} \abs{(\Pc \mi)(\rho \zeta)}\,\dd \beta(\zeta)\\
		&=\norm{\mi}_{\Mc^1(\bD)}
	\end{split}
	\]
	by monotone convergence, Proposition~\ref{prop:1}, and Remark~\ref{oss:13} (since $\abs{\Pc\mi}$ is subharmonic by~\cite[Corollary 2.1.14]{Krantz}). 
	Then, $(\Pc \mi)^{(\xi)}$ belongs to the \emph{harmonic} Hardy space $\Hc^1(\Db_\xi)$ for almost every $\xi\in \widehat{\bD}$, so that there is a measure $\mi_\xi$ on $\bD$, supported in $\pi^{-1}(\xi)$, such that 
	\[
	(\Pc \mi)(z)=\int_{\bD} \frac{1-\abs{z}^2}{\abs{\zeta-z}^2}\,\dd \mi_\xi(\zeta)
	\]
	for every $z\in \Db_\xi$.   In addition, Proposition~\ref{prop:1} and the previous remarks show that
	\[
	\int_{\widehat{\bD}} \norm{\mi_\xi}_{\Mc^1(\bD)}\,\dd \widehat\beta(\xi)=\int_{\widehat {\bD}} \sup_{0<\rho<1} \norm{ (\Pc \mi)^{(\xi)}_\rho }_{L^1(\beta_\xi)}\,\dd \widehat\beta(\xi)= \norm{\mi}_{\Mc^1(\bD)}.
	\]
	Furthermore, Proposition~\ref{prop:1} shows that $(\Pc \mi)_\rho \cdot \beta_\xi$ converges vaguely to $\mi_\xi$ for $\rho\to 1^-$, for $\widehat \beta$-almost every $\xi\in \widehat{\bD}$. Analogously, $(\Pc \mi)_\rho\cdot \beta$ converges vaguely to $\mi$ for $\rho\to 1^-$. Therefore, for every $f\in C(\bD)$,
	\[
	\begin{split}
		\int_{\bD} f\,\dd \mi&=\lim_{\rho \to 1^-}\int_{\bD} f (\Pc \mi)_\rho\,\dd \beta\\
			&= \lim_{\rho \to 1^-}\int_{\widehat{\bD}} \int_{\bD}  f (\Pc \mi)_\rho\,\dd \beta_\xi\,\dd\widehat \beta(\xi)\\
			&= \int_{\widehat{\bD}} \int_{\bD} f  \,\dd \mi_\xi\,\dd\widehat \beta(\xi),
	\end{split}
	\]
	where  the last equality follows from the dominated convergence theorem, since 
	\[
	\abs*{\int_{\bD} f (\Pc \mi)_\rho\,\dd \beta_\xi}=\abs*{\int_{\bD} \int_{\bD} \frac{1-\rho^2}{\abs{\zeta-\rho \zeta'}^2} f(\zeta') \,\dd \beta_\xi(\zeta') \,\dd \mi_\xi(\zeta) }\meg \norm{f}_{L^\infty(\beta)}  \norm{\mi_\xi}_{\Mc^1(\bD)}
	\]
	for $\widehat \beta$-almost every $\xi\in \widehat{\bD}$,
	and
	\[
	\int_{\widehat {\bD}}\norm{\mi_\xi}_{\Mc^1(\bD)}\,\widehat \beta(\xi)=  \norm{\mi}_{\Mc^1(\bD)}
	\]
	by the previous remarks. Now, let $N$ be the set of $\zeta\in \bD$ such that $\mi_{\pi(\zeta)}$ is defined and vanishes. Then, $(\Pc \mi)(\rho \zeta)=0$ for every $\rho\in [0,1)$ and for every $\zeta\in N$. Assume by contradiction that $\beta(N)>0$. Since  $(\Pc \mi)_\rho$ is real analytic, this implies that $\Pc\mi$ vanishes on $\rho \bD$ for every $\rho\in (0,1)$. Since $ (\Pc \mi)(\rho \,\cdot\,)$ is holomorphic on a neighbourhood of $\Cl(D)$ and since $\bD$ is the \v Silov boundary of $D$, this implies that $\Pc (\mi)$ vanishes on $\rho D$ for every $\rho\in (0,1)$, so that $\Pc(\mi)=0$ and $\mi=0$: contradiction. Then, $N$ is $\beta$-negligible.	
	The fact that $\widehat \beta$ is a pseudo-image measure of $\mi$ under $\pi$ and that $(\mi_\xi)$ is actually a disintegration of $\mi$ then follows from Propositions~\ref{prop:6} and~\ref{prop:7}.
	
	Finally, the fact that, if $\mi$ is holomorphic, then almost every $\mi_\xi$ is holomorphic (hence absolutely continuous with respect to $\beta_\xi$) follows from the fact that $\Pc (\mi)$ is then holomorphic, so that also $(\Pc \mi)^{(\xi)}$ is holomorphic for every $\xi\in \widehat{\bD}$.
\end{proof}

\begin{cor}
	Let $\mi$ be a holomorphic measure on $\bD$. Then, $\mi$ is absolutely continuous (with respect to $\beta$).
\end{cor}

\begin{proof}
	This follows from Remarks~\ref{oss:3} and~\ref{oss:13} and Theorem~\ref{prop:8}.
\end{proof}

\begin{cor}\label{cor:2}
	Let $\mi$ be a pluriharmonic measure on $\bD$. Then, $\mi$ is absolutely continuous with respect to $\Hc^{m-1}$, where $m$ denotes the dimension of the (real) analytic manifold $\bD$.\footnote{Notice that $m$ is also the Hausdorff dimension relative to the distance induced by the Euclidean distance (with respect to which Hausdorff measures are considered here). This would \emph{not} be the Hausdorff dimension of $\bD$ if $\bD$ were endowed with a sub-Riemannian distance associated with its CR structure.}
\end{cor}

In particular, $\mi$ cannot have any atoms unless $n=1$. Notice that this does \emph{not} mean that $\mi$ has a density with respect to $\Hc^{m-1}$, since this latter measure is far from being $\sigma$-finite on $\bD$.

\begin{proof}
	Let $E$ be an $\Hc^{m-1}$-negligible subset of $\bD$.
	Observe first that, by a simple extension of~\cite[remarks following Theorem 2]{BesicovitchMoran}, one may prove that $E\times \T$ is $\Hc^{m}$-negligible (with respect to the distance $((\zeta,\alpha),(\zeta',\alpha'))\mapsto \max(\abs{\zeta-\zeta'}, \abs{\alpha-\alpha'})$, for example). Since the mapping $\bD\times \T\ni (\zeta,\alpha)\mapsto \alpha \zeta\in \bD$ is real analytic, hence Lipschitz, this proves that $\T E=\pi^{-1}(\pi(E))$ is $\Hc^{m}$-negligible, hence $\beta$-negligible.
	Consequently, $\widehat \beta(\pi(E))=0$, so that $E$ is $\mi$-negligible since either $\mi=0$ or $\widehat \beta$ is a pseudo-image measure of $\mi$ under $\pi$ by Theorem~\ref{prop:8}.
\end{proof}

\begin{prop}\label{prop:12}
	Let $\mi$ be a non-zero positive pluriharmonic measure on $\bD$. Then, $\mi$ has a vaguely continuous disintegration $(\mi_\xi)$ relative to its pseudo-image measure $\widehat \beta$.
	In addition,
	\[
	(\Pc \mi)(z)= \int_{\bD} \frac{1-\abs{z}^2}{\abs{\zeta-z}^2}\,\dd \mi_\xi(\zeta)
	\]
	for every $\xi\in \widehat{\bD}$ and  for every $z\in \Db_\xi$.
\end{prop}

\begin{proof}
	By Proposition~\ref{prop:5}, for every $\zeta\in \bD$ there is a unique positive measure $\mi'_\zeta$ on $\T$ such that 
	\[
	(\Pc \mi)(w \zeta)=\int_\T \frac{1-\abs{w}^2}{\abs{\alpha-w}^2}\,\dd \mi'_\zeta(\alpha)
	\]
	for every $w\in \Db$. Set $\pi_\zeta\colon \T\ni \alpha \mapsto \alpha \zeta\in \bD$, and define $\mi''_\zeta\coloneqq (\pi_\zeta)_*(\mi'_\zeta)$. Then, 
	\[
	(\Pc \mi)(w \zeta)=\int_{\bD} \frac{1-\abs{w}^2}{\abs{\zeta'-w\zeta}^2}\,\dd \mi''_\zeta(\zeta')
	\]
	for every $w\in \Db$, so that $\mi''_{\zeta}=\mi''_{\alpha \zeta}$ for every $\alpha\in \T$. In addition, if $\xi=\pi(\zeta)$ and $\mi_\xi\coloneqq \mi''_\zeta$, then $(\mi_\xi)$ is a disintegration of $\mi$ relative to its pseudo-image measure $\widehat \beta$ by Theorem~\ref{prop:8}. It will then suffice to show that the mapping $\bD\ni \zeta \mapsto \mi''_\zeta\in \Mc^1(\bD)$ is vaguely continuous. In other to do that, observe first that $\norm{\mi'_\zeta}_{\Mc^1(\T)}= (\Pc \mi)(0)=\norm{\mi}_{\Mc^1(\bD)}$ for every $\zeta\in \bD$. The fact that the mapping $\bD\ni \zeta \mapsto (\Pc\mi)(w \zeta)\in \C$ is continuous for every $w\in \Db$ and the density of the vector space generated by the functions $\frac{1-\abs{w}^2}{\abs{\,\cdot\,-w}^2}$, $w\in \Db$, in $C(\T)$ then imply that the mapping $\bD\ni \zeta \mapsto \mi'_\zeta\in \Mc^1(\T)$ is vaguely continuous. Then, take $f\in C(\bD)$, and observe that
	\[
	\int_{\bD} f\,\dd \mi''_\zeta= \int_\T f(\alpha \zeta)\,\dd \mi'_\zeta(\alpha)
	\]
	for every $\zeta\in \bD$. Observe that the mapping $\bD\ni \zeta\mapsto f(\,\cdot\,\zeta)\in C(\T)$ is continuous  and that the mapping $\bD\ni \zeta \mapsto \mi'_\zeta\in \Mc^1(\T)$ is vaguely continuous (with $\norm{\mi'_\zeta}_{\Mc^1(\T)}=\norm{\mi}_{\Mc^1(\bD)}$ for every $\zeta\in \bD$). Since $\bD$ is compact, this is sufficient to prove that the mapping 
	\[
	\bD\ni\zeta\mapsto \int_\T f(\alpha \zeta)\,\dd \mi'_\zeta(\alpha) =\int_{\bD} f\,\dd \mi''_\zeta\in \C
	\]
	is continuous. By the arbitrariness of $f$, this completes the proof.
\end{proof} 

\begin{prop}\label{prop:9}
	Let $\mi$ be a non-zero pluriharmonic measure on $\bD$, and let $(\mi_\xi)$ be a disintegration of $\mi$ relative to its pseudo-image measure $\widehat \beta$ under $\pi$. Then, for $\widehat \beta$-almost every $\xi\in \widehat{ \bD}$ and for every $z\in \Db_\xi$,
	\[
	(\Cc \mi)(z)= \int_{\bD} \frac{1}{1-\langle z\vert \zeta\rangle}\,\dd \mi_\xi(\zeta).
	\] 
\end{prop}

\begin{proof}
	Assume first that $\mi\in \Pc_\R\cdot \beta$. Then, Proposition~\ref{prop:4} shows that $\mi=(P+\overline Q)\cdot \beta$ for some  $P,Q\in \Pc_\C$, so that  $\mi_\xi=(P+\overline Q)\cdot \beta_\xi$ for $\widehat \beta$-almost every $\xi\in \widehat{\bD}$ by Remark~\ref{oss:13}. We may also assume that $Q(0)=0$, so that
	\[
	(\Cc \mi)(z)=P(z)= \int_{\bD} \frac{1}{1-\langle z\vert \zeta\rangle}\,\dd \mi_\xi(\zeta)
	\]
	for $\widehat \beta$-almost every $\xi\in \widehat {\bD}$ and for every $z\in \Db_\xi$. Then, the assertion is proved in this case.
	
	Now, consider the general case. By Proposition~\ref{prop:1} and Theorem~\ref{prop:8}, $\mi$ and $\mi_\xi$ are the vague limits of $(\Pc\mi)_\rho\cdot \beta$ and $(\Pc\mi)_\rho\cdot \beta_\xi$, respectively, for $\rho\to 1^-$ (for $\widehat \beta$-almost every $\xi\in \widehat{\bD}$). In addition, by Proposition~\ref{prop:4} there is a filter $\Ff$ on $(\Pc_\C+\overline{\Pc_\C})\cdot \beta$ which converges vaguely to $\mi$. Therefore, 
	$(\Pc P)_\rho $  converges uniformly to $(\Pc\mi)_\rho$ as $P$ runs along $\Ff$, for every $\rho\in (0,1)$, so that
	\[
	[\Cc(\Pc \mi)_\rho](z)=\int_{\bD} \frac{(\Pc \mi)_\rho (\zeta)}{1-\langle z\vert \zeta\rangle}   \beta_\xi(\zeta)
	\]
	for $\widehat \beta$-almost every $\xi\in \widehat{\bD}$ and for every $z\in \Db_\xi$, by uniform convergence. Finally, letting $\rho\to 1^-$ we see that
	\[
	(\Cc \mi)(z)=\int_{\bD} \frac{1}{1-\langle z\vert \zeta\rangle}   \mi_\xi(\zeta)
	\]
	for $\widehat \beta$-almost every $\xi\in \widehat{\bD}$ and for every $z\in \Db_\xi$, by vague convergence. 
\end{proof}

The following result is an example of the properties of measures on $\T$ which may be transferred to pluriharmonic measures on $\bD$ by means of disintegrations. The classical result is konwn as Poltoratski's distribution theorem (cf.~\cite{Poltoratski}). Cf.~\cite[Proposition 2.12]{AleksandrovDoubtsov} for the case of the unit ball.

\begin{prop}\label{prop:10}
	Let $\mi$ be a pluriharmonic measure on $\bD$. Then,
	\[
	\lim_{y\to + \infty} \pi y \chi_{\Set{\zeta \in \bD\colon \abs{(\Cc \mi)_1(\zeta)}>y  }}\cdot \beta = \abs{\mi^s}
	\]
	in the vague topology, where $\mi^s$ denotes the singular part of $\mi$ with respect to $\beta$.
\end{prop}

In particular,
\[
\lim_{y\to + \infty} \pi y \beta(\Set{\zeta \in \bD\colon \abs{(\Cc \mi)_1(\zeta)}>y  })= \norm{\mi^s}_{\Mc^1(\bD)}.
\]

\begin{proof} 
	We may assume that $\mi$ is non-zero. Take a disintegration $(\mi_\xi)$ of $\mi$ relative to its pseudo-image measure $\widehat \beta$, by Theorem~\ref{prop:8}.
	Observe that, by Proposition~\ref{prop:9}, it is clear that $(\Cc\mi)_1$ is well defined $\beta$-almost everywhere, and that
	\[
	(\Cc \mi)_1(\zeta)= (\Cc_{\xi} \mi_\xi)_1(\zeta)
	\]
	for almost every $\xi\in \widehat{\bD}$, and for $\beta_\xi$-almost every  $\zeta\in \bD$, where $\Cc_{\xi}$ denotes the Cauchy integral on $\pi^{-1}(\xi)$ (considered as the boundary of $\Db_\xi$). Therefore,
	\[
	\int_{\bD} f(\zeta)\chi_{\Set{\zeta\in \bD\colon \abs{(\Cc \mi)_1(\zeta)}>y}}\,\dd \beta(\zeta)=\int_{\widehat {\bD}} \int_{\bD} f \chi_{\Set{\zeta\in \pi^{-1}(\xi)\colon \abs{(\Cc_{\xi} \mi_\xi)_1(\zeta)}>y}}\,\dd \beta_\xi\,\dd \widehat \beta(\xi)
	\]
	for every $f\in C(\bD)$, by Remark~\ref{oss:13}. Now, observe that the mapping $\nu\mapsto (\Cc_{\T}\nu)_1$ maps $\Mc^1(\T)$ into $L^{1,\infty}(\T)$, and is continuous by the closed graph theorem (cf.~\cite[Theorem 1]{Poltoratski}), where $\Cc_\T$ denotes the Cauchy integral on $\T$. Consequently, there is a constant $C>0$ such that
	\[
	y\abs*{\int_{\bD} f \chi_{\Set{\zeta\in \pi^{-1}(\xi)\colon \abs{(\Cc_{\xi} \mi_\xi)_1(\zeta)>y}}}\,\dd \beta_\xi}\meg C \norm{f}_{L^\infty(\bD)} \norm{\mi_\xi}_{\Mc^1(\bD)}
	\]
	for every $y>0$ and for every $f\in C(\bD)$. Since the function $\xi\mapsto\norm{\mi_\xi}_{\Mc^1(\bD)} $ is $\widehat \beta$-integrable, by the dominated convergence theorem and~\cite[Theorem 1]{Poltoratski} we see that, for every $f\in C(\bD)$, 
	\[
	\lim_{y\to +\infty} y\pi \int_{\bD} f(\zeta)\chi_{\Set{\zeta\in \bD\colon \abs{(\Cc \mi)_1(\zeta)}>y}}\,\dd \beta(\zeta)=\int_{\widehat {\bD}} \int_{\bD} f \,\dd \abs{\mi_\xi^s}\,\dd \widehat \beta(\xi)=\int_{\bD} f\,\dd \abs{\mi^s},
	\] 
	where $\mi_\xi^s$ denotes the singular part of $\mi_\xi$ with respect to $\beta_\xi$ (cf.~Remark~\ref{oss:3}). The proof is complete.	
\end{proof}

\section{Clark Measures}\label{sec:4}

\begin{deff}
	Take a holomorphic function $\phi\colon D\to \Db$. Then, Proposition~\ref{prop:5} shows that, for every $\alpha\in \T$, there is a unique positive pluriharmonic measure $\mi_\alpha$ on $\bD$ such that
	\[
	\Pc(\mi_\alpha)=\Re \left( \frac{\alpha+\phi}{\alpha-\phi} \right)= \frac{1-\abs{\phi}^2}{\abs{\alpha-\phi}^2}
	\]
	or, equivalently,
	\[
	\frac{\alpha+\phi}{\alpha-\phi} - i \Im\left( \frac{\alpha+\phi(0)}{\alpha-\phi(0)} \right)= \Hc(\mi_\alpha).
	\]
	We shall also write $\mi_\alpha[\phi]$ instead of $\mi_\alpha$.
\end{deff}

\begin{oss}\label{oss:6}
	Notice that, if $\mi_{\alpha,\xi}$ denotes the Clark measure corresponding to the restriction of $\phi$ to $\Db_\xi$ (considered as a measure on $\bD$), for every $\xi\in \widehat { \bD}$, then  $(\mi_{\alpha,\xi})_{\xi}$ is a vaguely continuous disintegration of $\mi_\alpha$ relative to $\widehat \beta$, thanks to Proposition~\ref{prop:12}.
\end{oss}

\begin{oss}\label{oss:5}
	Let $\mi$ be a non-zero positive pluriharmonic measure on $\bD$, and take $\alpha\in \T$. Then, there is a unique holomorphic function $\psi\colon D\to \Db$ such that $\psi(0)\in \R \alpha $ and $\mi=\mi_\alpha[\psi]$.
	
	In particular, $\mi$ is singular with respect to $\beta$ if and only if $\psi$ is inner.
\end{oss}

\begin{proof}
	By Proposition~\ref{prop:5}, $\Hc(\mi)$ is a holomorphic function, is strictly positive at $0$, and maps $D$ into the right half-plane $\R_+^*+i \R=-i \C_+$. Now, $h\colon \Db\ni w \mapsto \frac{\alpha+w}{\alpha-w}\in -i \C_+$ is a biholomorphism with inverse $w\mapsto \alpha \frac{w-1}{w+1}$. 
	Thus, $\psi\coloneqq h^{-1}\circ \Hc(\mi)$ is a holomorphic function from $D$ into $\Db$, $\psi(0)\in \R \alpha$,  $\frac{\alpha+\psi}{\alpha-\psi}=\Hc(\mi)$, and $\frac{\alpha+\psi(0)}{\alpha-\psi(0)}=(\Hc\mi)(0)>0$, so that  $\mi=\mi_\alpha[\psi]$. The fact that $\psi$ is unique follows from the fact that $h\circ \psi$ is determined by the relation $\Re (h\circ \psi)=\Pc(\mi)$.
	
	The second assertion follows from Proposition~\ref{prop:1}, the equality $\Re \Hc(\mi)=\Pc(\mi)$, and the fact that $h$ induces homeomorphisms of $\C\setminus \Set{\alpha}$ onto $\C\setminus \Set{-1}$ and of $\T\setminus \Set{\alpha}$ onto $i\R$. 
\end{proof}

\begin{lem}\label{lem:4}
	For every $\alpha\in \T$,
	\[
	\norm{\mi_\alpha}_{\Mc^1(\bD)}=\Re \left( \frac{\alpha+\phi(0)}{\alpha-\phi(0)} \right)= \frac{1-\abs{\phi(0)}^2}{\abs{\alpha-\phi(0)}^2}.
	\]
\end{lem}

Cf.~\cite[Proposition 9.1.8]{CimaMathesonRoss}.

\begin{proof}
	This follows from the equality $\Pc(\mi_\alpha)(0)=\norm{\mi_\alpha}_{\Mc^1(\bD)}$.
\end{proof}

\begin{lem}\label{lem:2}
	For every $\alpha\in \T$,
	\[
	\Cc(\mi_\alpha)=\frac{1}{1-\overline \alpha \phi}+  \frac{\alpha\overline{\phi(0)}}{1-\alpha\overline{\phi(0)}}
	\]
\end{lem}

Cf.~\cite[Corollary 9.1.7]{CimaMathesonRoss}.

\begin{proof}
	This follows from the equality $\Hc(\mi_\alpha)=2 \Cc(\mi_\alpha)-\norm{\mi_\alpha}_{\Mc^1(\bD)}$.
\end{proof}

\begin{prop}\label{prop:2}
	For every $\alpha\in \T$, the absolutely continuous part of $\mi_\alpha$ (with respect to $\beta$) has density
	\[
	\Re \left( \frac{\alpha+\phi_1}{\alpha-\phi_1} \right)= \frac{1-\abs{\phi_1}^2}{\abs{\alpha-\phi_1}^2}.
	\]
	In particular, $\phi$ is an inner function if and only if $\mi_\alpha$ is singular with respect to $\beta$ for some (or, equivalently, every) $\alpha\in \T$.
\end{prop}

Cf.~\cite[Proposition 9.1.14 and Corollary 9.1.15]{CimaMathesonRoss}.

\begin{proof}
	This follows from Proposition~\ref{prop:1}.
\end{proof}

\begin{prop}\label{prop:13}
	The following hold:
	\begin{enumerate}
		\item[\textnormal{(1)}] for every $\alpha\in \T$, the singular part $\sigma_\alpha$ of $\mi_\alpha$ is concentrated in $\phi_1^{-1}(\alpha)$;
		
		\item[\textnormal{(2)}]  for every $\alpha,\alpha'\in \T$ with $\alpha\neq \alpha'$, $\sigma_\alpha$ and $\sigma_{\alpha'}$ are singular with respect to each other.
	\end{enumerate}
\end{prop}

Cf.~\cite[Corollary 9.1.24]{CimaMathesonRoss}.
 
\begin{proof}
	Take $\alpha\in \T$ and let $(\mi_{\alpha,\xi})$ be the vaguely continuous disintegration of $\mi_\alpha$ relative to $\widehat \beta$  (cf.~Proposition~\ref{prop:12} and Remark~\ref{oss:6}). Then,~\cite[Corollary 9.1.24]{CimaMathesonRoss} shows that the singular part $\sigma_{\alpha,\xi}$ of $\mi_{\alpha,\xi}$ with respect to $\beta_\xi$ is concentrated on $\pi^{-1}(\xi)\cap \phi_1^{-1}(\alpha)$ for every $\xi\in \widehat{\bD}$ (cf.~Remark~\ref{oss:6}).  
	By Remark~\ref{oss:3}, this is sufficient to conclude that $\sigma_\alpha(\bD\setminus \phi_1^{-1}(\alpha))=0$, whence (1). Assertion (2) then follows from the fact that $\sigma_\alpha$ and $\sigma_{\alpha'}$ are concentrated on disjoint sets if $\alpha\neq\alpha'$.
\end{proof}

\begin{prop}
	Assume that $\phi(0)=0$, and take $\alpha\in \T$ and $k\Meg 1$. Then,
	\[
	\Cc^{(k)}(\mi_\alpha)=\sum_{h=1}^k \overline \alpha^h \Cc^{(k)}(\phi_1^h).
	\]
	In particular, for every homogeneous $P\in \Pc_\C$ of degree $k$,
	\[
	\int_{\bD} \overline P\,\dd \mi_\alpha= \sum_{h=1}^k \overline \alpha^h \int_{\bD} \phi_1(\zeta)^h \overline{P(\zeta)}\,\dd \zeta
	\]
	and, by conjugation,
	\[
	\int_{\bD}  P\,\dd \mi_\alpha= \sum_{h=1}^k  \alpha^h \int_{\bD} \overline{\phi_1(\zeta)}^h {P(\zeta)}\,\dd \zeta
	\]
\end{prop}

Cf.~\cite[Proposition 9.1.9]{CimaMathesonRoss}.

Notice that, since $\int_{\bD} \,\dd  \mi_\alpha =1$ by Lemma~\ref{lem:4}, and since $\mi_\alpha$ is positive and pluriharmonic, the integrals $\int_{\bD} \overline P\,\dd \mi_\alpha$, for all non-constant homogeneous $P\in \Pc_\C$, completely determine $\mi_\alpha$.

\begin{proof}
	Using Remark~\ref{oss:16}, it is readily verified that $\Cc^{(k)}(\mi_\alpha)$ is the homogeneous component of degree $k$ of the holomorphic function (cf.~Lemma~\ref{lem:2})
	\[
	\Cc(\mi_\alpha)=\frac{1}{1-\overline \alpha \phi}.
	\]
	Now,
	\[
	\frac{1}{1-\overline \alpha \phi}= \sum_{h \Meg 0} \overline \alpha^h \phi^h
	\]
	uniformly on the compact subsets of $ D$,
	so that
	\[
	\Cc^{(k)}(\mi_\alpha)=\sum_{h=1}^k \overline \alpha^h \Cc^{(k)}(\phi_1^h)
	\]
	since $\phi(0)=0$.
\end{proof}

\begin{teo}\label{teo:1}
	The mapping $\alpha \mapsto \mi_\alpha$ is vaguely continuous, and
	\[
	\int_{\T} \mi_\alpha\,\dd \beta_\T(\alpha)= \beta.
	\]
	If, in addition, $\phi$ is inner, then $\beta_\T$ is a pseudo-image measure of $\beta$ under $\phi_1$, and $(\mi_\alpha)$ is a disintegration of $\beta$ relative to $\beta_\T$. Finally, if $\phi(0)=0$, then $(\phi_1)_*(\beta)=\beta_\T$.
\end{teo}

Cf.~\cite[Theorem 9.3.2 and Proposition 9.4.2]{CimaMathesonRoss}.
Notice that~\cite[Proposition 2 and Theorem 1 of Chapter V, \S\ 3]{BourbakiInt1} then implies that, if $f\colon \bD\to \C$ is $\beta$-integrable, then $f$ is $\mi_\alpha$-integrable for $\beta_\T$-almost every $\alpha\in \T$, the mapping $\alpha\mapsto \int_{\bD} f\,\dd \mi_\alpha$ is $\beta_\T$-integrable, and
\[
\int_\T \int_{\bD} f\,\dd \mi_\alpha\,\dd \beta_\T(\alpha)=\int_{\bD} f\,\dd \beta.
\]

\begin{proof}
	Observe first that clearly the mapping
	\[
	\alpha \mapsto \langle \mi_\alpha, \Pc(z,\,\cdot\,)\rangle = \Pc(\mi_\alpha)(z)= \Re\left(\frac{\alpha+\phi(z)}{\alpha-\phi(z)}  \right)
	\]
	is continuous on $\T$ for every $z\in D$. 
	In addition, Proposition~\ref{prop:1} shows that the set of the $\Pc(z,\,\cdot\,)$, $z\in D$, is total in $C(\bD)$, since its polar in $\Mc^1(\bD)$ is $\Set{0}$. In addition,
	\[
	\norm{\mi_\alpha}_{\Mc^1(\bD)}=\frac{1-\abs{\phi(0)}^2}{\abs{\alpha-\phi(0)}^2}\meg \frac{1-\abs{\phi(0)}^2}{(1-\abs{\phi(0)})^2}=\frac{1+\abs{\phi(0)}}{1-\abs{\phi(0)}}
	\]
	for every $\alpha\in \T$, thanks to Lemma~\ref{lem:4}. This is sufficient to prove that the mapping $\alpha \mapsto \mi_\alpha$ is vaguely continuous.
	
	By the previous remarks, the second assertion will be established once we prove that
	\[
	\int_\T \langle\mi_\alpha, \Pc(z,\,\cdot\,)\rangle\,\dd\beta_\T(\alpha)= \int_{\bD} \Pc(z,\zeta)\,\dd \beta(\zeta)
	\]
	for every $z\in D$.
	To this aim, observe that Proposition~\ref{prop:1} shows that
	\[
	\int_\T \langle\mi_\alpha, \Pc(z,\,\cdot\,)\rangle\,\dd\beta_\T(\alpha)=\int_\T \frac{1-\abs{\phi(z)}^2}{\abs{\alpha-\phi(z)}^2} \,\dd \beta_\T(\alpha)=1=\int_{\bD}\Pc(z,\zeta)\,\dd\beta( \zeta),
	\]
	so that the second assertion follows.
	
	Now, assume that $\phi$ is inner. Then, Proposition~\ref{prop:13} shows that $\mi_\alpha$ is concentrated in $\phi_1^{-1}(\alpha)$ for every $\alpha\in \T$. Since $\mi_\alpha\neq 0$ for every $\alpha\in \T$ by Lemma~\ref{lem:4}, Proposition~\ref{prop:7} shows that 
	$\beta_\T$ is a pseudo-image measure of $\beta$ under $\phi_1$, and that $(\mi_\alpha)$ is a disintegration of $\beta$ relative to $\beta_\T$. If, in addition,  $\phi(0)=0$, then $\mi_\alpha$ is a probability measure for every $\alpha\in \T$ by Lemma~\ref{lem:4}, so that $\beta_\T$ is actually the image measure of $\beta$ under $\phi_1$.
\end{proof}

\begin{prop}\label{prop:16}
	Let $\psi\colon \Db \to \Db $ be a holomorphic map. Then, for every $\alpha\in \T$,
	\[
	\mi_{\alpha}[\psi\circ \phi]=\int_{ \T} \mi_{\alpha'}[\phi]\,\dd \mi_\alpha[\psi](\alpha').
	\]
\end{prop}

\begin{proof}
	Observe first that the mapping $\alpha'\mapsto \mi_{\alpha'}[\phi]$ is vaguely continuous by Theorem~\ref{teo:1}, so that the positive measure $\nu\coloneqq\int_{ \T} \mi_{\alpha'}[\phi]\,\dd \mi_\alpha[\psi](\alpha')$ is well defined (and pluriharmonic by Proposition~\ref{prop:4}). Then, observe that
	\[
	\begin{split}
		(\Pc \mi_{\alpha}[\psi\circ \phi])(z)&= \frac{1-\abs{\psi(\phi(z))}^2}{\abs{\alpha- \psi(\phi(z))}^2}\\
		&=\int_{\T} \frac{1- \abs{\phi(z)}^2}{\abs{\alpha'-\phi(z)}^2}\,\dd \mi_{\alpha}[\psi](\alpha')\\
		&=\int_{\T} \int_{\bD}  \Pc(z,\zeta)  \,\dd \mi_{\alpha'}[\phi](\zeta)\,\dd \mi_{\alpha}[\psi](\alpha')\\
		&= (\Pc \nu)(z)
	\end{split}
	\]
	for every $z\in D$. This completes the proof.
\end{proof}

\begin{cor}\label{cor:3}
	Let $\psi$ be a biholomorphism of $\Db$. Then, for every $\alpha\in \T$,
	\[
	\mi_{\psi(\alpha)}[\psi\circ \phi]=\frac{1}{\abs{\psi'(\alpha)}}\mi_{\alpha}[\phi].
	\]
\end{cor}

Cf.~\cite[Introduction]{Poltoratski2}. In particular, this result allows to reduce to the case in which $\phi(0)=0$ in various situations.

\begin{proof}
	It suffices to apply Proposition~\ref{prop:16}, since $\mi_{\psi(\alpha)}[\psi]=\frac{1}{\abs{\psi'(\alpha)}} \delta_\alpha$ (cf.~Proposition~\ref{prop:13} and~\cite[Theorem 9.2.1]{CimaMathesonRoss}).
\end{proof}

\begin{prop}\label{prop:3}
	For every $\alpha\in \T$ and for every $z,z'\in D$,
	\begin{equation}\label{eq:3}
	\langle \Cc(z,\,\cdot\,)\vert \Cc(z',\,\cdot\,)\rangle_{L^2(\mi_\alpha)}=\frac{1-\phi(z)\overline{\phi(z')}}{(1-\overline \alpha \phi(z))(1-\alpha \overline{\phi(z')})}\Cc(z,z')
	\end{equation}
	In particular, the mapping
	\[
	\Cc_\phi\colon D\times D\ni (z,z')\mapsto (1-\phi(z)\overline{\phi(z')})\Cc(z,z')\in \C,
	\] 
	is a positive kernel.
\end{prop}

Cf.~\cite[Proposition 3.5]{AleksandrovDoubtsov}.

\begin{proof}
	Let us first prove that the mapping $D\ni z \mapsto \Cc(z,\,\cdot\,) \in L^2(\mi_\alpha)$ is holomorphic. In fact, the mapping $D\ni z \mapsto \Cc(z,\,\cdot\,)\in C(\bD)$ is holomorphic since $\Cc(\mi)=[z\mapsto \langle \mi,\Cc(z,\,\cdot\,)\rangle]$ is holomorphic for every $\mi\in \Mc^1(\bD)$.  Therefore, both sides of~\eqref{eq:3} are sesquiholomorphic functions of $z,z'\in D$. Hence, it will suffice to prove~\eqref{eq:3} when $z=z'$ (cf.~\cite[Theorem 1 of Chap.\ V, \S\ 21, Sect.\ 66]{Shabat}). In this case,
	\[
	\begin{split}
	\norm{\Cc(z,\,\cdot\,)}_{L^2(\mi_\alpha)}^2&=\int_{\bD} \abs{\Cc(z,\zeta)}^2\,\dd \mi_\alpha(\zeta)\\
		&=\Cc(z,z) (\Pc \mi_\alpha)(z)\\
		&=\Cc(z,z) \frac{1 -\abs{\phi(z)}^2}{\abs{\alpha-\phi(z)}^2}\\
		&=\frac{1- \abs{\phi(z)}^2}{\abs{1-\overline \alpha\phi(z)}^2} \Cc(z,z),
	\end{split}
	\]
	whence the first assertion.
	
	Conserning the second assertion, we have to prove that $\sum_{z,z'\in D} a_z \overline{a_{z'}} \Cc_\phi(z,z')\Meg 0$ for every $(a_z)\in \C^{(D)}$. However,
	\[
	\sum_{z,z'\in D} a_z \overline{a_{z'}} \Cc_\phi(z,z')= \sum_{z,z'\in D} a'_z \overline {a'_{z'}} \langle \Cc(z,\,\cdot\,)\vert \Cc(z',\,\cdot\,)\rangle_{L^2(\mi_\alpha)}= \norm*{\sum_{z\in D} a'_z\Cc(\,\cdot\,,z)}_{L^2(\mi_\alpha)}^2\Meg 0
	\]
	where $a'_z\coloneqq (1-\overline \alpha \phi(z)) a_z$ for every $z\in D$. One may also see this fact as a consequence, e.g., of~\cite[Lemma 3.9]{Timotin}.
\end{proof}

\begin{deff}
	We define $\Cc_\phi$ as in Proposition~\ref{prop:3}, and $\phi^*(H^2(D))$  as the unique Hilbert space with reproducing kernel $\Cc_\phi$.\footnote{Thus, $\phi^*(H^2(D))$ is the completion of the vector space generated by the $\Cc_\phi(\,\cdot\,,z)$, $z\in D$, with respect to the scalar product such that $\langle \Cc_\phi(\,\cdot\,,z)\vert \Cc_\phi(\,\cdot\,,z')\rangle=\Cc_\phi(z',z)$ for every $z,z'\in D$. }
	We also define, for every $\alpha\in \T$,
	\[
	\Cc_{\phi,\alpha}(f)\coloneqq (1-\overline\alpha \phi)\Cc(f\cdot \mi_\alpha)
	\]
	for every $f\in L^1(\mi_\alpha)$.
	
	Furthermore, we define $H^2(\mi_\alpha)$ as the closure of $\Pc_\C$ (or, equivalently, of the vector space generated by the $\Cc(\,\cdot\,,z)$, $z\in D$, cf.~Remark~\ref{oss:4}) in $L^2(\mi_\alpha)$.
\end{deff}

The space $\phi^*(H^2(D))$ is thus an analogue of the de Branges--Rovnyak spaces, of model spaces when $\phi$ is inner. Notice that, when $D$ is the unit ball,  a different possible analogue of model spaces has been proposed in~\cite{AleksandrovDoubtsov} (alongside the present one).

\begin{oss}
	Notice that $\Cc-\Cc_\phi= (\phi\otimes \overline\phi)\Cc$ is the reproducing kernel of $\phi H^2(D)$, endowed with the norm $\phi f \mapsto \norm{f}_{H^2(D)}$ (unless $\phi=0$), so that the linear mapping $\phi H^2(D)\times \phi^*(H^2(D))\ni (f_1,f_2)\mapsto f_1+f_2\in H^2(D)$ is a surjective partial isometry. In addition, the following hold:
	\begin{itemize}
		\item   $\phi^*(H^2(D))$ embeds contractively into $ H^2(D)$ (cf.~\cite[Lemma 3.9]{Timotin});

		\item $\phi^*(H^2(D))$ embeds isometrically into $H^2(D)$ (necessarily onto $H^2(D)\ominus \phi H^2(D)$) if and only if $\phi$ is inner.
	\end{itemize}
\end{oss}

\begin{oss}
	Notice that, when $\phi(0)=0$,
	\[
	\Cc_{\phi,\alpha}(f)=\frac{\Cc(f\cdot \mi_\alpha)}{\Cc(\mi_\alpha)},
	\]
	thanks to Lemma~\ref{lem:2}.
\end{oss}

\begin{teo}\label{cor:1}
	The mapping  $\Cc_{\phi,\alpha}$ induces a partial isometry of $L^2(\mi_\alpha)$ onto $\phi^*(H^2(D))$, with kernel $L^2(\mi_\alpha)\ominus H^2(\mi_\alpha)$. In other words, $f\in \ker \Cc_{\phi,\alpha}$ if and only if $\Cc(f \cdot \mi_\alpha)=0$. Furthermore,
	\[
		\Cc_{\phi,\alpha}(\Cc(\,\cdot\,,z))=\frac{1}{1-\alpha \overline{\phi(z)}} \Cc_\phi(\,\cdot\,,z).
	\]
	for every $\alpha\in \T$ and for every $z\in D$. 
\end{teo}

Cf.~\cite[Corollary 9.5.2]{CimaMathesonRoss},~\cite[Theorem 5.1]{AleksandrovDoubtsov}, and~\cite[Theorem 3.1 to Corollary 3.4]{AleksandrovDoubtsov2}.

Notice that, in general, $H^2(\mi_\alpha)\neq L^2(\mi_\alpha)$ even when $\phi$ is inner. See~\cite[Remark 3.5]{AleksandrovDoubtsov2} for a simple counterexample when $D$ is a polydisc. In general, similar couterexamples arise on every reducible domain (cf.~the proof of Proposition~\ref{prop:17}).

\begin{proof}
	Observe that, by Proposition~\ref{prop:3},
	\[
	\begin{split}
	\Cc_{\phi,\alpha}(\Cc(\,\cdot\,,z))(z')&=(1-\overline\alpha \phi(z'))\langle \Cc(z',\,\cdot\,)\vert \Cc(z,\,\cdot\,)\rangle_{L^2(\mi_\alpha)}\\
		&=\frac{(1-\phi(z')\overline{\phi(z)})}{1-\alpha \overline{\phi(z)}} \Cc(z',z)\\
		&=\frac{1}{1-\alpha \overline{\phi(z)}} \Cc_\phi(z',z)
	\end{split}
	\]
	for every $z,z'\in D$. Therefore,
	\[
	\langle \Cc_{\phi,\alpha}(\Cc(\,\cdot\,,z))\vert \Cc_{\phi,\alpha}(\Cc(\,\cdot\,,z))\rangle_{\phi^*(H^2(D))}=\frac{1}{(1-\alpha \overline{\phi(z)})(1-\overline \alpha \phi(z'))} \Cc_\phi(z',z)=\langle \Cc(z,\,\cdot\,)\vert \Cc(z',\,\cdot\,)\rangle_{L^2(\mi_\alpha)}
	\]
	for every $z,z'\in D$, thanks to   Proposition~\ref{prop:3}. Consequently, $\Cc_{\phi,\alpha}$ induces an isometry of   $H^2(\mi_\alpha)$ onto   $\phi^*(H^2(D))$. 
	To conclude, observe that, given $f\in L^2(\mi_\alpha)$, one has $\Cc_{\phi,\alpha}(f)=0$ if and only if
	\[
	\Cc(f\cdot \mi_\alpha)(z)=\langle f\vert \Cc(\,\cdot\,,z)\rangle_{L^2(\mi_\alpha)}=0
	\]
	for every $z\in D$, that is, if and only if $f\in L^2(\mi_\alpha)\ominus H^2(\mi_\alpha)$.  
\end{proof}

\begin{prop}\label{prop:33}
	Denote by $\sigma_\alpha $ the singular part of $\mi_\alpha$ with respect to $\beta$, and  take $f\in \phi^{*}(H^2(D))$. Then, $\Cc_{\phi,\alpha}^* f=f_1$ $\sigma_\alpha$-almost everywhere for $\beta_\T$-almost every $\alpha\in \T$. If, in addition, $\phi$ is inner and $f_1\in C(\bD)$, then $\Cc_{\phi,\alpha}^* f=f_1$ $\mi_\alpha$-almost everywhere for \emph{every} $\alpha\in \T$.
\end{prop}

\begin{proof}
	Observe first that, if $f=\Cc_\phi(\,\cdot\,,z)$ for some $z\in D$, then Theorem~\ref{cor:1} shows that
	\[
	\Cc_{\phi,\alpha}^* f = (1-\alpha \overline{\phi(z)}) \Cc(\,\cdot\,,z)=f_1
	\]
	$\mi_\alpha$-almost everywhere, where the last equality follows from the fact that $\phi_1=\alpha$ $\sigma_\alpha$-almost everywhere thanks to Proposition~\ref{prop:13}. 
	The same then holds for every $f$ in the vector space $V$ generated by the $\Cc_\phi(\,\cdot\,,z)$, $z\in D$, which is (contained and) dense in $\phi^*(H^2(D))$. 
	
	Now, take an arbitrary $f\in \phi^{*}(H^2(D))$. Then, there is a sequence $(f^{(j)})$ of elements of $V$ which converges to $f$ in $\phi^*(H^2(D))$.  Then, $f^{(j)}_1=\Cc_{\phi,\alpha}^* f^{(j)}$ converges to $\Cc_{\phi,\alpha}^* f$ in $L^2(\mi_\alpha)$ for every $\alpha\in \T$. 
	Notice that, since $f_1$ is Borel measurable and $\beta$-almost everywhere defined (and in $L^2(\beta)$), it is in particular $\mi_\alpha$-measurable and almost everywhere defined (and in $L^2(\mi_\alpha)$) for $\beta_\T$-almost every $\alpha\in \T$, thanks to Theorem~\ref{teo:1} and~\cite[Theorem 1 of Chapter V, \S\ 3, No.\ 3]{BourbakiInt1}.
	In addition,  $f^{(j)}_1$ converges to $f_1$ in $L^2(\beta)$, so that
	\[
	\int_{\T} \liminf_{j\to \infty}\norm{ f^{(j)}_1-  f_1}_{L^2(\mi_\alpha)}^2\,\dd \beta_\T(\alpha)\meg \lim_{j\to \infty}\int_{\T} \norm{ f^{(j)}_1-  f_1}_{L^2(\mi_\alpha)}^2\,\dd \beta_\T(\alpha)=\lim_{j\to \infty} \int_{\bD} \abs{f^{(j)}- f}^2\,\dd \beta=0
	\]
	by Fatou's lemma.
	Consequently, for $\beta_\T$-almost every $\alpha\in \T$ there is a subsequence of $f^{(j)}_1$ which converges to $ f_1$ in $L^2(\mi_\alpha)$. Since $f^{(j)}_1=\Cc_{\phi,\alpha}^* f^{(j)}$ $\sigma_\alpha$-almost everywhere, and $\Cc_{\phi,\alpha}^* f^{(j)}$ converges to $\Cc_{\phi,\alpha}^* f$ in $L^2(\mi_\alpha)$ for every $\alpha\in \T$, this proves that $\Cc_{\phi,\alpha}^* f=f_1$ $\sigma_\alpha$-almost everywhere for $\beta_\T$-almost every $\alpha\in \T$.

	Now, assume that $\phi$ is inner and that $f_1\in C(\bD)$. 
	Then, Remark~\ref{oss:4} shows that $f_1$ belongs to the closure of $\Pc_\C$ in $C(\bD)$, hence to the closure of $\Pc_\C$ in $L^2(\mi_\alpha)$ for every $\alpha\in \T$. In other words, $f_1\in H^2(\mi_\alpha)$ for every $\alpha\in \T$. Hence, our assertion will follow in we prove that $\Cc_{\phi,\alpha} f_1=f$ for every $\alpha\in \T$. By the previous remarks, this holds for $\beta_\T$-almost every $\alpha\in \T$. In addition, the mapping
	\[
	\T\ni \alpha \mapsto (\Cc_{\phi,\alpha} f_1)(z)=(1-\overline \alpha \phi(z)) \int_{\bD}  \Cc(z,\zeta) f_1(\zeta)\,\dd \mi_\alpha(\zeta) \in \C
	\]
	is continuous for every $z\in D$, thanks to  Theorem~\ref{teo:1} and Proposition~\ref{prop:1}. Consequently,  $\Cc_{\phi,\alpha} f_1=f$ for \emph{every} $\alpha\in \T$, as claimed.
\end{proof}

\begin{deff}
	Let $\mi$ be a Radon measure on $\bD$. Then, for every $p\in [1,\infty]$, we denote by $\PM^p(\mi)$ the space of $f\in L^p(\mi)$ such that $f\cdot \mi$ is pluriharmonic.
	
	In addition, for every $\alpha\in \T$ we define 	$\phi_{*}(H^2)$   as the space of $f\in H^2(D)$ such that  $f\vert_{\Db_\xi}\in(\phi\vert_{\Db_\xi})^*(H^2(\Db_\xi))$ for $\widehat \beta$-almost every $\xi\in \widehat{\bD}$.
\end{deff}

\begin{oss}
	Let $\mi$ be a Radon measure on $\bD$. Then, $\PM^p(\mi)$ is a closed subspace of $L^p(\mi)$ for every $p\in [1,\infty]$.
\end{oss}

\begin{proof}
	It suffices to observe that, by Proposition~\ref{prop:4}, $\PM^p(\mi)$ is the set of $f\in L^p(\mi)$ such that $\int_{\bD} f P\,\dd \mi=0$ for every $P\in \Pc_\R\cap (\Pc_\C\cup \overline{\Pc_\C})^\perp$.
\end{proof}

\begin{oss}\label{oss:10}
	Take $\alpha\in \T$ and a biholomorphism $\psi$ of $\Db$. Then,
	\[
	\Cc_{\psi(\alpha), \psi\circ \phi}(f)= \frac{1-\overline{\psi(\alpha)} (\psi\circ \phi)}{\abs{\psi'(\alpha)}}\Cc(f\cdot \mi_\alpha)=\frac{1-\overline{\psi(\alpha)} (\psi\circ \phi)}{\abs{\psi'(\alpha)} (1-\overline \alpha \phi)}\Cc_{\phi,\alpha}(f)
	\]
	for every $f\in L^1(\mi_\alpha)$ (cf.~Corollary~\ref{cor:3}). In addition, $\psi$ extends to a holomorphic mapping on a neighbourhood of $\Cl(\Db)$, so that the mapping $(\alpha,w)\mapsto \begin{cases} \frac{1-\overline{\psi(\alpha)} \psi(w)}{ 1-\overline \alpha w} &\text{if $w\neq \alpha$}\\ \frac{\alpha\psi'(\alpha)}{\psi(\alpha)} &\text{if $w=\alpha$} \end{cases}  $ is continuous on $\T\times \Cl(\Db)$. Hence, there is a constant $C>0$ such that
	\[
	\frac{1}{C}\meg \abs*{\frac{1-\overline{\psi(\alpha)} (\psi\circ \phi)}{ (1-\overline \alpha \phi)}} \meg C
	\]
	on $D$, for every $\alpha\in \T$.
\end{oss}

The following  result is a simple consequence of~\cite{Poltoratski2}.

\begin{lem}\label{lem:9}
	Assume that $D=\Db$. Take $\alpha\in \T$ and $f\in L^2(\mi_\alpha)$. Then, the following hold:
	\begin{enumerate}
		\item[\textnormal{(1)}] $(\Cc_{\phi,\alpha} f)_\rho$ converges to $(\Cc_{\phi,\alpha} f)_1$ pointwise $\mi_\alpha$-almost everywhere and in $L^2(\mi_\alpha)$;
		
		\item[\textnormal{(2)}] $(\Cc_{\phi,\alpha} f)_1=f$ $\sigma_\alpha$-almost everywhere, where $\sigma_\alpha$ denotes the singular part of $\mi_\alpha $ (with respect to $\beta_\T$);
		
		\item[\textnormal{(3)}] there is a constant $C>0$, which depends only on $\phi(0)$, such that for every $\rho \in [0,1]$, 
		\[
		\norm{(\Cc_{\phi,\alpha} f)_\rho}_{L^2(\mi_\alpha)}\meg C \norm{f}_{L^2(\mi_\alpha)}.
		\]
		One may take $C=2$ if $\phi(0)=0$.
		\end{enumerate}
\end{lem}

\begin{proof}
	Notice that, by means of Remark~\ref{oss:10}, we may reduce to the case in which $\phi(0)=0$.
	Then, define $S_k\coloneqq \sum_{j=0}^k \frac{1}{j!} (\Cc_{\phi,\alpha} f)^{(j)}(0) (\,\cdot\,)^j$, so that~\cite[Lemma 2.2 and the Remark which follows]{Poltoratski2} show that
	\[
	\norm{S_k}_{L^2(\mi_\alpha)}\meg 2\norm{f}_{L^2(\mi_\alpha)}
	\]
	for every $k\in\N$. In addition,~\cite[Theorem 2.5]{Poltoratski2} shows that $\sum_k S_k$ converges in $L^2(\mi_\alpha)$, while~\cite[Theorem  2.7]{Poltoratski2} shows that $(\Cc_{\phi,\alpha} f)_\rho$ converges to $(\Cc_{\phi,\alpha} f)_1$ pointwise $\mi_\alpha$-almost everywhere, and that $(\Cc_{\phi,\alpha} f)_1=f$ $\sigma_\alpha$-almost everywhere. We have thus proved (2) and the first half of (1).
	
	Now, observe that by summation by parts,
	\[
	\begin{split}
		\norm{(\Cc_{\phi,\alpha} f)_\rho}_{L^2(\mi_\alpha)}&=\lim_{k\to \infty} \norm*{ \sum_{j=0}^k \rho^j \frac{1}{j!} (\Cc_{\phi,\alpha} f)^{(j)}(0) (\,\cdot\,)^j}_{L^2(\mi_\alpha)}\\
		&=\lim_{k\to \infty}\norm*{ \sum_{j=0}^k  S_j (\rho^j-\rho^{j+1})+ \rho^{k+1}S_k}_{L^2(\mi_\alpha)}\\
		&\meg 2 \norm{f}_{L^2(\mi_\alpha)}
	\end{split}
	\]
	for every $\rho\in [0,1)$, whence (3). In particular, $(\Cc_{\phi,\alpha} f)_1\in L^2(\mi_\alpha)$. 
	It only remains to prove that $(\Cc_{\phi,\alpha} f)_\rho $ converges to $(\Cc_{\phi,\alpha} f)_1$ in $L^2(\mi_\alpha)$. Observe that it will suffice to show this latter fact only for $f\in H^2(\mi_\alpha)$. Since, in addition, the operators $H^2(\mi_\alpha)\ni f \mapsto (\Cc_{\phi,\alpha} f)_\rho\in L^2(\mi_\alpha)$, $\rho\in [0,1]$, are equicontinuous, it will suffice to show that $(\Cc_{\phi,\alpha} \Cc(\,\cdot\,,z))_\rho$ converges to $(\Cc_{\phi,\alpha} \Cc(\,\cdot\,,z))_1$ in $L^2(\mi_\alpha)$, for $\rho \to 1^-$, for every $z\in \Db$. 
	By Theorem~\ref{cor:1}, this is equivalent to saying that $[\Cc_\phi(\,\cdot\,,z)]_\rho$ converges to $[\Cc_\phi(\,\cdot\,,z)]_1$, for $\rho \to 1^-$, in $L^2(\mi_\alpha)$. 
	Now, $[\Cc(\,\cdot\,,z)]_\rho$ converges to $[\Cc(\,\cdot\,,z)]_1$, for $\rho \to 1^-$, in $C(\bD)$ by Proposition~\ref{prop:1}, whereas $1-\phi_\rho \overline{\phi(z)}$ is uniformly bounded and converges $\mi_\alpha$-almost everywhere to $1-\phi_1 \overline{\phi(z)}$ for $\rho \to 1^-$ (thanks to Proposition~\ref{prop:13}), so that the assertion follows.
\end{proof}

\begin{prop}\label{prop:34}
	Assume that $\phi$ is inner. Take $\alpha\in \T$ and $f\in \phi_*(H^2)$. Then, the following hold:
	\begin{enumerate}
		\item[\textnormal{(1)}]  $f\in \phi^*(H^2)$;
		
		\item[\textnormal{(2)}] $f_\rho$ converges to $f_1$ pointwise  $\mi_\alpha$-almost everywhere and in $L^2(\mi_\alpha)$, for $\rho \to 1^-$;
		
		\item[\textnormal{(3)}] for $\widehat \beta$-almost every $\xi\in \widehat {\bD}$ and for every $z\in \Db_\xi$,
		\[
		f(z)= (1-\overline \alpha \phi(z)) \int_{\bD} \frac{f_1(\zeta)}{1-\langle z\vert \zeta \rangle}\,\dd \mi_{\alpha,\xi}(\zeta) .
		\]
	\end{enumerate}
\end{prop}

\begin{proof}
	(1) Since $\phi$ is inner, it will suffice to show that $\langle f\vert \phi P\rangle_{H^2(D)}=0$ for every $P\in \Pc_\C$. Then, observe that
	\[
	\langle f\vert \phi P\rangle_{H^2(D)}=\int_{\widehat{\bD}} \langle f\vert _{\Db_\xi}\vert (\phi P)\vert_{\Db_\xi}\rangle_{H^2(\Db_\xi)}\,\dd \widehat \beta(\xi)=0,
	\]
	whence our assertion.
	
	(2)--(3)  	Observe that Proposition~\ref{prop:2} and Lemma~\ref{lem:9} show  that $f_\rho=[\Cc_{\phi\vert \Db_\xi, \alpha}(f_1\vert \pi^{-1}(\xi))]_\rho$ converges to $[\Cc_{\phi\vert \Db_\xi, \alpha}(f_1\vert \pi^{-1}(\xi))]_1=f_1$ poinwise $\mi_{\alpha,\xi}$-almost everywhere and in $L^2(\mi_{\alpha,\xi})$  for $\widehat\beta$-almost every $\xi\in \widehat{\bD}$  (with some abuse of notation). This is sufficient to prove that $f_\rho$ converges to $f_1$ pointwise  $\mi_\alpha$-almost everywhere, as well as (3). For what concerns the remainder of (2), observe that Lemma~\ref{lem:9} shows  that there is a constant $C>0$ such that
	\[
	\norm{f_\rho-f_1}_{L^2(\mi_{\alpha,\xi})}\meg C \norm{ f_1 }_{L^2(\mi_{\alpha,\xi})}=C \norm{f}_{H^2(\Db_\xi)}=C \norm{f_1}_{L^2(\beta_\xi)}
	\]
	for $\widehat \beta$-almost every $\xi\in \widehat{\bD}$ (where the first equality follows from the fact that $\phi\vert_{\Db_\xi}$ is inner for $\widehat \beta$-almost every $\xi\in \widehat{\bD}$, in which case $\Cc_{\phi\vert \Db_\xi,\alpha}$ is an isometry from $L^2(\mi_{\alpha,\xi})$ into $(\phi\vert\Db_\xi)^*(H^2(\Db_\xi))$ by~\cite[Lemma 9.5.3]{CimaMathesonRoss}, and the latter space in turn embeds isometrically into $H^2(\Db_\xi)$).
	The convergence of $f_\rho $ to $f_1$ in $L^2(\mi_\alpha)$ then follows by means of the dominated convergence theorem.
\end{proof}

\begin{prop}\label{prop:35}
	Take $\alpha\in \T$. Then, the following hold:
	\begin{enumerate}
		\item[\textnormal{(1)}] $\Cc_{\phi,\alpha}(\PM^2(\mi_\alpha))\subseteq \phi_*(H^2)$;
		
		\item[\textnormal{(2)}] if $f\in \PM^2(\mi_\alpha)$, then $(\Cc_{\phi,\alpha} f)_\rho$ converges to $(\Cc_{\phi,\alpha} f)_1$ pointwise $\mi_\alpha$-almost everywhere and in $L^2(\mi_\alpha)$;
		
		\item[\textnormal{(3)}] if $f\in \PM^2(\mi_\alpha)$, then $(\Cc_{\phi,\alpha} f)_1=f$ $\sigma_\alpha$-almost everywhere, where $\sigma_\alpha$ denotes the singular part of $\mi_\alpha$ with respect to $\beta$;

		\item[\textnormal{(4)}] if $\phi$ is inner, then $\PM^2(\mi_\alpha)\subseteq H^2(\mi_\alpha)$, that is, $\Cc_{\phi,\alpha}^* \Cc_{\phi,\alpha}$ is the identity on $\PM^2(\mi_\alpha)$;
		
		\item[\textnormal{(5)}] if $\phi$ is inner, then  $\Cc_{\phi,\alpha}^* f=f_1$ $\mi_\alpha$-almost everywhere for every $f\in \Cc_{\phi,\alpha}(\PM^2(\mi_\alpha))$;

		\item[\textnormal{(6)}] if $\phi$ is inner and $f\in \Cc_{\phi,\alpha}(\PM^2(\mi_\alpha))$, then there is $g\in H^2(D)$ such that $g(0)=0$ and $f_1=\phi_1 \overline g_1$ $\beta$-almost everywhere.
	\end{enumerate}
\end{prop}

If $D$ is the unit ball, then the converse of (6) holds, cf.~\cite[Theorem 1.1]{AleksandrovDoubtsov}.

\begin{proof}
	(1) This follows from the fact that, by Proposition~\ref{prop:9},
	\[
	(\Cc_{\phi,\alpha} f)(z)=(1-\overline \alpha \phi(z)) \int_{\bD} \frac{f(\zeta)}{1-\langle z\vert \zeta \rangle}\,\dd \mi_{\alpha,\xi}(\zeta)=\Cc_{\phi\vert_{\Db_\xi},\alpha}(f\vert \pi^{-1}(\xi))(z)
	\]
	for $\widehat \beta$-almost every $\xi\in \widehat{\bD}$ and for every $ z\in \Db_\xi$.
	
	(2)--(3) Take $f\in \PM^2(\mi_\alpha)$. By Proposition~\ref{prop:9},
	\[
	\Cc_{\phi,\alpha}(f)(z)= (1-\overline \alpha\phi(z)) \Cc(f\cdot \mi_\alpha)(z)=(1-\overline \alpha\phi(z)) \int_{\bD} \frac{f(\zeta)}{1-\langle z\vert \zeta\rangle}\,\dd \mi_{\alpha,\xi}(\zeta)
	\]
	for $\widehat \beta$-almost every $\xi\in \widehat{\bD}$ and for every $z\in \Db_\xi$. 
	Therefore, by means of Lemma~\ref{lem:9} we see that, for $\widehat \beta$-almost every $\xi\in \widehat{\bD}$, $(\Cc_{\phi,\alpha}f)_\rho$ converges to $ (\Cc_{\phi,\alpha} f)_1$  $\mi_{\alpha,\xi}$-almost everywhere and in $L^2(\mi_{\alpha,\xi})$  for $\rho\to 1^-$, and that $(\Cc_{\phi,\alpha}f)_1=f$ $\sigma_{\alpha,\xi}$-almost everywhere, where $\sigma_{\alpha,\xi}$ denotes the singular part of $\mi_{\alpha,\xi}$ with respect to $\beta_\xi$. 
	This is sufficient to deduce that $(\Cc_{\phi,\alpha} f)_\rho$ converges to $(\Cc_{\phi,\alpha} f)_1$ pointwise $\mi_\alpha$-almost everywhere, as well as (3) (cf.~Remark~\ref{oss:3}). 
	For what concerns the convergence in $L^2(\mi_\alpha)$, observe that Lemma~\ref{lem:9} and the previous remarks show  that     there is a constant $C>0$ such that
	\[
	\int_{\bD} \abs{(\Cc_{\phi,\alpha} f)_\rho-(\Cc_{\phi,\alpha}f)_1   }^2\,\dd \mi_{\alpha,\xi}\meg C \norm{f}_{L^2(\mi_{\alpha,\xi})}^2
	\]
	for $\widehat\beta$-almost every $\xi\in \widehat{\bD}$. Consequently, the dominated convergence theorem shows that
	\[
	\lim_{\rho\to 1^-}\norm{(\Cc_{\phi,\alpha}f )_\rho-(\Cc_{\phi,\alpha}f )_1}_{L^2(\mi_\alpha)}^2=\lim_{\rho\to 1^-}\int_{\widehat{ \bD}} \int_{\bD} \abs{(\Cc_{\phi,\alpha}f )_\rho-(\Cc_{\phi,\alpha}f )_1  }^2\,\dd \mi_{\alpha,\xi}(\zeta)\,\dd \widehat \beta(\xi)=0.
	\]

	(4)--(5) Assume that $\phi$ is inner and take $f\in \PM^2(\mi_\alpha)$. Then, (2) and (3) show that $(\Cc_{\phi,\alpha} f)_\rho\to f$ in $L^2(\mi_\alpha)$ for $\rho\to 1^-$. 
	In addition, $(\Cc_{\phi,\alpha} f)_\rho$ belongs to the closure of $\Pc_\C$ in $C(\bD)$ since   $(\Cc_{\phi,\alpha} f)(\rho\,\cdot\,)$ is holomorphic on $(1/\rho) D$, for every $\rho\in (0,1)$ (cf.~Remark~\ref{oss:15}). 
	Hence, $(\Cc_{\phi,\alpha} f)_\rho\in H^2(\mi_\alpha)$ for every $\rho\in (0,1)$, so that $f\in H^2(\mi_\alpha)$. 
	Since $\Cc_{\phi,\alpha}^*\Cc_{\phi,\alpha}$ is the canonical projection of $L^2(\mi_\alpha)$ onto $H^2(\mi_\alpha)$, the second assertion of (4) follows, as well as (5).

	(6)  Assume that $\phi$ is inner and take $f\in \Cc_{\phi,\alpha}(\PM^2(\mi_\alpha))$. Observe that $\int_{\bD} f_1\,\dd \mi_\alpha= \Cc (f_1\cdot \mi_\alpha)(0)=\frac{1}{1-\overline \phi(0)} (\Cc_{\phi,\alpha}f_1)(0)=\frac{f(0)}{1-\overline \phi(0)}$ by (4) and (5). Then, define 
	\[
	h\coloneqq \int_{\bD} \overline f_1\,\dd \mi_\alpha -\Cc(\overline f_1\cdot \mi_\alpha)=\frac{\overline {f(0)}}{1-\alpha \overline {\phi(0)}}-\Cc(\overline f_1\cdot \mi_\alpha)
	\]
	and 
	\[
	g= (\phi-\alpha)h,
	\]
	so that $g_1=\phi_1(1-\alpha \overline \phi_1)h_1$.
	Then, $h(0)=0$ and $g(0)=0$. In addition, since $\overline f_1\cdot \mi_\alpha$ is pluriharmonic by (4) and (5), Proposition~\ref{prop:9} shows that
	\[
	h(z)= \int_{\bD} \overline{f}\,\dd \mi_{\alpha,\xi}-\int_{\bD} \frac{\overline{f(\zeta)}}{1-\langle z\vert \zeta\rangle}\,\dd \mi_{\alpha,\xi}(\zeta)
	\]
	for $\widehat \beta$-almost every $\xi\in \widehat{\bD}$ and for every $z\in \Db_\xi$. 
	Then, e.g.~\cite[the proof of Theorem 1.1]{AleksandrovDoubtsov} shows that the restriction of $g$ to $\Db_\xi$ belongs to $H^2(\Db_\xi)$,  and that $f_1=\phi_1\overline g_1$ $\beta_\xi$-almost everywhere for $\widehat \beta$-almost every $\xi\in \widehat {\bD}$. Therefore, $f_1=\phi_1\overline g_1$ $\beta$-almost everywhere, and
	\[
	\begin{split}
	\sup_{\rho \in (0,1)} \int_{\bD} \abs{g_\rho}^2\,\dd \beta&=\int_{\widehat {\bD}} \lim_{\rho\to 1^-} \int_{\bD} \abs{g_\rho}^2\,\dd \beta_\xi\,\dd \widehat \beta(\xi)\\
		&=\int_{\widehat{\bD}} \int_{\bD} \abs{g_1}^2\,\dd \beta_\xi\,\dd \widehat \beta(\xi)\\
		&=\int_{\widehat{\bD}} \int_{\bD} \abs{f_1}^2\,\dd \beta_\xi\,\dd \widehat \beta(\xi)\\
		&=\int_{\bD} \abs{f_1}^2\,\dd \beta
	\end{split}
	\]
	by monotone convergence, so that $g\in H^2(D)$.
\end{proof}

\section{Henkin Measures}\label{sec:5}

\begin{deff}
	Denote by $A(D)$ the space of continuous functions on $D\cup \bD$ whose restriction to $D$ is holomorphic.
	We denote by $A(D)^\perp$ the polar of $A(D)$ in $\Mc^1(\bD)$, that is, the space of Radon measures $\mi$ on $\bD$ such that $\int_{\bD} f\,\dd \mi=0$ for every $f\in A(D)$.
	
	A Radon measure $\mi$ on $\bD$ is a Henkin measure (or an $A(D)^\perp$-measure) if 
	\[
	\lim_{j\to \infty}\int_{\bD} f^{(j)}\,\dd \mi=0
	\] 
	for every bounded sequence $(f^{(j)})$ in $A(D)$ which converges pointwise (hence locally uniformly) to $0$ on $D$.
	
	A positive Radon measure $\mi$ on $\bD$ is   a representing measure if $\int_{\bD} f\,\dd \mi=f(0)$ for every $f\in A(D)$.
	
	A Radon measure $\mi$ on $\bD$ is totally singular if it is singular with respect to every representing measure on $\bD$.
\end{deff}

Notice that every representing measure $\mi$ is a probability measure, since $\norm{\mi}_{\Mc^1(\bD)}=\int_{\bD}1\,\dd \mi=1$. In addition, total singularity does not change if one considers representing measure associated with different points in $D$, as in the case of the unit ball (cf.~\cite[9.1.3]{Rudin}).

\begin{oss}
	Take $z_1,z_2\in D$ and a positive Radon measure $\mi_1$ on $\bD$ such that $\int_{\bD} f\,\dd \mi=f(z_1)$ for every $f\in A(D)$. Then, there is a positive Radon measure $\mi_2$ on $\bD$ such that $\int_{\bD} f\,\dd \mi_2=f(z_2)$ for every $f\in A(D)$ and such that $\mi_1$ is absolutely continuous with respect to $\mi_2$.
\end{oss}

\begin{proof}
	By Proposition~\ref{prop:1}, there is $c>0$ such that $c\Pc(z_1,\,\cdot\,)\meg \Pc(z_2,\,\cdot\,)$ on $\bD$. Then, set $\mi_2\coloneqq [\Pc(z_2,\,\cdot\,)-c\Pc(z_1,\,\cdot\,)]\cdot \beta +c \mi_1$, so that $\mi_2$ is a positive Radon measure on $\bD$. In addition, if $f\in A(D)$, then
	\[
	\int_{\bD} f\,\dd \mi_2=(\Pc f_1)(z_2)-c (\Pc f_1)(z_1)+c\int_{\bD}f\,\dd \mi_1=f(z_2).
	\]
	Finally, it is clear that $\mi_1$ is absolutely continuous with respect to $\mi_2$, since $\mi_1\meg \mi_2/c$.
\end{proof}

\begin{oss}\label{oss:18}
	For every $\rho\in (0,1)$ and for every $z,z'\in \Cl(D)$,
	\[
	\Cc(\rho z, z')=\overline{\Cc(\rho z',z)}.
	\]
	In addition, for every $\rho \in (0,1)$ and for every $\zeta,\zeta'\in \bD$,
	\[
	\Pc(\rho \zeta,\zeta')=\Pc(\rho\zeta',\zeta).
	\]
\end{oss}

\begin{proof}
	By Remark~\ref{oss:16},
	\[
	\Cc(\rho z, z')=\sum_{k\in \N} \Cc^{(k)}(\rho z, z')=\sum_{k\in \N} \rho^k \Cc^{(k)}(z,z')=\sum_{k\in \N} \Cc^{(k)}(z,\rho z')=\Cc(z,\rho z')
	\]
	when $z,z'\in D$, whence the first assertion by Proposition~\ref{prop:1}. For the second assertion, observe that
	\[
	\Pc(\rho \zeta,\zeta')=\frac{\abs{\Cc(\rho\zeta,\zeta')}^2}{\Cc(\rho \zeta,\rho \zeta)}=\frac{\abs{\Cc(\zeta,\rho\zeta')}^2}{\Cc(\rho \zeta',\rho \zeta')}=\Pc(\rho \zeta',\zeta),
	\]
	where the second equality follows from the fact that there is $k\in K$ such that $\zeta=k \zeta'$, and $\Cc(\rho k \zeta', \rho k \zeta')=\Cc(\rho \zeta',\rho \zeta')$ since $k$ is linear and $\Cc$ is $k\times k$-invariant (as composition with $k$ induces an isometry of $H^2(D)$).
\end{proof}

\begin{prop}\label{prop:14}
	Let $\mi$ be a Henkin measure on $\bD$, and take $\eps>0$. Then,   there are $f\in L^1(\beta)$ and $\nu\in A(D)^\perp$ such that $\mi=\nu+f\cdot \beta$ and $\norm{f}_{L^1(\beta)}\meg \norm{\mi}_{\Mc^1(\bD)}+\eps$.
\end{prop}

This is the so-called Valskii decomposition. The proof is formally identical to that of~\cite[Theorem 9.2.1]{Rudin} and will not be repeated (one has to make use of Remark~\ref{oss:18}).

\begin{deff}
	We say that $D$ satisfies property (H) if evey Radon measure which is absolutely continuous with respect to a Henkin measure is itself a Henkin measure.
	
	We say that $D$ satisfies property (CL) if every Henkin measure is absolutely continuous with respect to a representing measure.
\end{deff}

Property (H) means that  `Henkin's theorem' holds on $D$, whereas property (CL) means that  `Cole--Range theorem' holds on $D$.

\begin{prop}\label{prop:17}
	Assume that $D$ satisfies property (H). Then, the following hold:
	\begin{enumerate}
		\item[\textnormal{(1)}] every positive pluriharmonic measure on $\bD$ which is singular to $\beta$ is totally singular;
		
		\item[\textnormal{(2)}] $D$ satisfies property (CL);
		
		\item[\textnormal{(3)}] if $\phi$ is a non-trivial inner function on $D$, then $H^2(\mi_\alpha[\phi])=L^2(\mi_\alpha[\phi] )$ for every $\alpha\in \T$;
		
		\item[\textnormal{(4)}] either $D$ is irreducible or admits no non-trivial inner functions.		
	\end{enumerate} 
\end{prop}

Cf.~\cite[Theorems 9.3.2 and 9.6.1]{Rudin} for (1) and (2),~\cite[Lemma 4.2]{AleksandrovDoubtsov} for (3), and~\cite[Remark 3.5]{AleksandrovDoubtsov2} for (4).
Observe that, as the proof shows,  (3) implies (4). In addition, the only domains were property (H) is known to hold (to the best of our knowledge) are the (Euclidean) unit balls, that is, those of rank $1$.

Notice, in addition, that every irreducible bounded symmetric domain which is either of tube type or of rank $1$ admits non-trivial inner functions, thanks to~\cite[Lemma 2.3]{KoranyiVagi} and~\cite{Aleksandrov}. 

\begin{proof}
	(1)--(2) The proof is analogous to that of~\cite[Theorems 9.3.2 and 9.6.1]{Rudin}.
	
	(3) The proof is analogous to that of~\cite[Lemma 4.2]{AleksandrovDoubtsov}. 
	
	(4) Assume by contradiction that $D$ is reducible and admits non-trivial inner functions. Then, using~\cite[Theorem 3.3 and its Corollary]{Koranyi2} as in the proof of~\cite[Theorem 3.3 (iii)]{KoranyiVagi}, we see that  there are two non-trivial (circular, convex) symmetric domains $D_1,D_2$ such that (up to a rotation) $D=D_1\times D_2$, and such that there is a non-trivial inner function $\psi$ on $D_1$.  Let $\pr_j\colon D\to D_j$ be the canonical projection for $j=1,2$. Observe that $\psi\circ \pr_1$ is an inner function on $D$, so that we may consider the positive pluriharmonic measures $\mi_1[\psi]$ on $\bD_1$ and $\mi_1[\psi\circ \pr_1]$ on $ \bD$. Observe that 
	\[
	\mi_1[\psi\circ \pr_1]=\mi_1[\psi]\otimes \beta_{\bD_2},
	\]
	where $\beta_{\bD_2}$ denotes the normalized invariant measure on $\bD_2$. Then, take a holomorphic polynomial $P$ on $D_2$ such that $P(0)=0$, and observe that
	\[
	\langle Q\vert \overline P\circ \pr_2\rangle_{L^2(\mi_1[\psi\circ \pr_1])}=\int_{\bD_1} \int_{\bD_2} Q(\zeta_1,\zeta_2) P(\zeta_2)\,\dd \beta_{\bD_2}(\zeta_2) \,\dd \mi_1[\psi](\zeta_1)=0
	\]
	for every $Q\in \Pc_\C$. Therefore, $\overline P\circ \pr_2\in L^2(\mi_1[\psi\circ \pr_1])\ominus H^2(\mi_1[\psi\circ \pr_1])$ and we may take $P$ so that $(P\circ \pr_2)\cdot \mi_1[\psi\circ \pr_1]=\mi_1[\psi]\otimes (\overline P \cdot\beta_{\bD_2})\neq 0$ since $D_2$ is non-trivial.
\end{proof}

\section{Composition Operators}\label{sec:6}

\begin{prop}\label{prop:15}
	Let $\phi\colon D\to \Db$ be a holomorphic function, and define $C_\phi\colon H^2(\Db)\ni f\mapsto f\circ \phi \in H^2(D)$. Then,\footnote{Here, $\norm{C_\phi}_{\Lin(H^2(\Db);H^2(D)),e}$ denotes the essential norm of $C_\phi$, that is, $\inf_T\norm{C_\phi+T}_{\Lin(H^2(\Db);H^2(D))}$, where $T$ runs through the set of compact operators from $H^2(\Db)$ into $H^2(D)$.}
	\[
	\norm{C_\phi}_{\Lin(H^2(\Db);H^2(D)),e}^2=\sup_{\alpha\in \T} \norm{\sigma_\alpha}_{\Mc^1(\bD)}=\Nf(\phi),
	\]
	where $\sigma_\alpha$ denotes the singular part of $\mi_\alpha[\phi]$ with respect to $\beta$, while
	\[
	\Nf(\phi)=\limsup_{\abs{z}\to 1^-} \int_{\bD} \frac{N_{\phi,\zeta}(z)}{1-\abs{z}}\,\dd \beta(\zeta)
	\]
	and
	\[
	N_{\phi,\zeta}(z)= - \sum_{z'\in \Db\colon \phi(z'\zeta)=z} \log \abs{z'}
	\]
	for every $z\in \Db \setminus \Set{\phi(0)}$.
\end{prop}

Cf.~\cite[Theorem 6.4]{AleksandrovDoubtsov}. The proof is formally identical to that of the cited reference (for the continuity, one has to observe that, if $f\in H^2(\Db)$ and $g$ is a harmonic majorant of $\abs{f}^2$, then $g\circ \phi$ is a \emph{pluriharmonic} majorant of $\abs{f\circ \phi}^2$, so that $f\circ \phi\in H^2(D)$; harmonicity, in this context, does not seem to suffice).

\section{Angular Derivatives}\label{sec:7}

\begin{deff}
	Let $k$ be the reproducing kernel of the Bergman space $A^2(D)=\Hol(D)\cap L^2(D)$, where $D$ is endowed with the Hausdorff measure $\Hc^n$, and denote by $h$ the Bergman metric, so that $h(z)(v,w)=\partial_v \overline{\partial_w} \log k(z,z)$ for every $z\in D$ and for every $v,w\in \C^n$. Denote by $B$ the Bergman operator, that, the unique sesquiholomorphic polynomial map from $\C^n\times \C^n$ into $ \Lc(\C^n)$ such that
	\[
	h(0)(v,w)=h(z)(B(z,z)v,w)
	\]
	for every $z\in D$ and for every $v,w\in \C^n$. For every $\zeta\in \bD$ and for every $a>0$, define the angular region
	\[
	D_a(\zeta)=\Set{ z\in D\colon \norm{ B(z,\zeta) P(\zeta,\zeta) }'^{1/2}<a (1-\norm{z}^2) },
	\]
	where $\norm{\,\cdot\,}$ denotes the norm on $\C^n$ whose open unit ball is $D$, $\norm{A}'=\max_{z\in \overline D} \norm{A z}$ for every $A\in \Lc(\C^n)$, and $P$ is the homogeneous component of degree $4$ of $B$.
\end{deff}

With the notation of Jordan triple systems, $B(z,z')=I-2 D(z,z')+Q(z)Q(z')$, so that $P(\zeta,\zeta)=Q(\zeta)^2$ is the orthogonal projector onto the Peirce space $Z_1(\zeta)$ associated with the maximal tripotent $\zeta$. In addition, $h(0)(v,w)=\tr D(v,w)$,   while $k(z,z')=( \Hc^n(D) \det B(z,z')  )^{-1}$ (cf.~\cite[Theorem 2.10]{Loos}, where some slightly different notation is used; we follow the conventions of~\cite{MackeyMellon}).

\begin{deff}
	Let $\Gamma\colon [0,1)\to D$ be a continuous curve, and take $\zeta\in \bD$. Set $\gamma\coloneqq \langle \Gamma\vert \zeta\rangle$, so that $\gamma \zeta$ is the orthogonal projection of $\Gamma$ on $\C \zeta$. The curve $\Gamma$ is said to be a $\zeta$-curve if $\lim\limits_{t\to 1^-} \Gamma(t)=\zeta$. If, in addition,
	\begin{equation}\label{eq:4}
	\lim_{t\to 1^-} \frac{\norm{\Gamma(t)-\gamma(t)}}{1-\abs{\gamma(t)}}=0,
	\end{equation}
	then $\Gamma$ is said to be special. Finally, if one assumes further that $\gamma$ is non-tangential (in $\Db$), then $\Gamma$ is said to be restricted.
\end{deff}

This terminology follows~\cite{MackeyMellon} (combined with~\cite{Calzi}). Notice that, when $D$ is the unit ball of $\C^n$ (endowed with the usual scalar product), then one may considerably relax~\eqref{eq:4} and simply require $\lim\limits_{t\to 1^-} \frac{\norm{\Gamma(t)-\gamma(t)}^2}{1-\abs{\gamma(t)}}=0 $, or, equivalently, $\lim\limits_{t\to 1^-} \frac{\norm{\Gamma(t)-\gamma(t)}^2}{1-\abs{\gamma(t)}^2}=0 $ (cf.~\cite[Sections 8.4 and 8.5]{Rudin}). Nonetheless, the result stated below would be incorrect for more general domains with this relaxed version of speciality (cf.~\cite[Example 4.6]{Calzi}).

\begin{prop}
	Let $\phi\colon D\to \Db$ be a holomorphic function, and take $\zeta\in \bD$. If $\mi_{\alpha, \pi(\zeta)}(\Set{\zeta})>0$ for some $\alpha\in \T$, then
	\[
	\lim_{\substack{z\to \zeta\\ z\in D_a(\zeta)}} \phi(z)=\alpha
	\]
	for every $a>0$, and
	\[
	\lim_{t\to + \infty} (\partial_\zeta \phi)(\Gamma(t))=\frac{\alpha}{\mi_{\alpha, \pi(\zeta)}(\Set{\zeta})}
	\]
	for every restricted special $\zeta$-curve $\Gamma$.
\end{prop}

Notice that here we made use of the `one-dimensional' Clark measure $\mi_{\alpha, \pi(\zeta)}$. The same considition is more difficult to formulate in terms of the `full' Clark measure $\mi_\alpha$.
Notice that also the converse holds: if both $f$ and $\partial_\zeta f$ have radial limits at $\zeta$ (in $\T$ and $\C\setminus \Set{0}$, respectively), then $\mi_{\alpha,\pi(\zeta)}(\Set{\zeta})>0$ (cf.~the proof of~\cite[Theorem 9.2.1]{CimaMathesonRoss}).

\begin{proof}
	Observe first that, by~\cite[Theorems 1.7.10 and  9.2.1 (and  the proof of the latter one)]{CimaMathesonRoss},
	\[
	\lim_{\rho \to 1^-} \phi(\rho \zeta)=\alpha,
	\]
	\[
	\lim_{\rho \to 1^-} \frac{1-\abs{\phi(\rho \zeta)}}{1-\rho} =\frac{1}{\mi_{\alpha, \pi(\zeta)}(\Set{\zeta})},
	\]
	and
	\[
	\lim_{\rho \to 1^-} \partial_\zeta \phi(\rho \zeta)=\frac{\alpha}{\mi_{\alpha, \pi(\zeta)}(\Set{\zeta})}.
	\]
	Therefore,~\cite[Lemma 3.3 and Theorem 3.7]{MackeyMellon} show that
	\[
	\lim_{\substack{z\to \zeta\\ z\in D_a(\zeta)}} \phi(z)=\alpha
	\]
	for every $a>0$. In addition,~\cite[Lemma 3.3, Corollary 4.5, and Theorem 4.6]{MackeyMellon} show that
	\[
	\lim_{t\to + \infty} \partial_\zeta \phi(\Gamma(t))=\frac{\alpha}{\mi_{\alpha, \pi(\zeta)}(\Set{\zeta})}
	\]
	for every restricted special $\zeta$-curve $\Gamma$.
\end{proof}

\section{Appendix: Disintegration}\label{sec:8}

\begin{deff}\label{def:1}
	Let $X,Y$ be two locally compact spaces with a countable base, and let $\mi$ and $\nu$ be complex Radon measures on $X$ and $Y$, respectively.  Let $p\colon X\to Y$ be a $\mi$-measurable map. Then, we say that $\nu$ is a (resp.\ weak) pseudo-image measure of $\mi$ under $p$ if  a subset $N$ of $Y$ is  $\nu$-negligible if and only if (resp.\ only if) $p^{-1}(N)$ is $\mi$-negligible.
	
	We denote by  $\Mc_+(X)$ the space of positive Radon measures on $X$.
\end{deff}

One may consider a more general setting, but the statements and proofs then become more complicated.

\begin{oss}
	Notice that, since $X$ has a countable base, there are $\mi$-measurable functions $f$ on $X$ such that $f\cdot \mi$ is a positive bounded measure and $f(x)\neq 0$ for every $x\in X$. For example, if $h\colon X\to \T$ is a density of $\mi$ with respect to $\abs{\mi}$ and $(K_j)$ is an increasing sequence of compact subsets of $X$ whose union is $X$, then one may take $f=\overline h\sum_{j} \frac{2^{-j}}{\abs{\mi}(K_j)+1}\chi_{K_j}$. Then, $\mi$ and $f\cdot \mi$ have the same negligible sets. In addition, the image measure $p_*(f\cdot \mi)$ is a pseudo-image measure of $\mi$ under $p$ in the sense of Definition~\ref{def:1}, thanks to~\cite[Corollary 2 to Proposition 2 of Chapter V, \S\ 6, No.\ 2]{BourbakiInt1}.
	Then, saying that $\nu  $ is a pseudo-image measure of $\mi$ under $p$ means that the Radon measures $\nu$ and  $p_*(f\cdot \mi)$ have the same negligible sets, that is, are equivalent, whereas saying that $\nu$ is a weak pseudo-image measure of $\mi$ under $p$ means that $p_*(f\cdot \mi)$ is absolutely continuous with respect to $\nu$.\footnote{The same assertion may be stated replacing $p_*(f\cdot \mi)$ with the not necessarily Radon measure $p_*(\abs{\mi})$. One the one hand, it is clear that, if $N$ is a $p_*(\abs{\mi})$-negligible subset of $Y$, then $p^{-1}(N)$ is $\mi$-negligible. Conversely, let $N$ be a subset of $Y$ such that $p^{-1}(N)$ is $\mi$-negligible. Then, $N$ is $p_*(f\cdot \mi)$-negligible, so that there is a $p_*(f\cdot \mi)$-negligible Borel subset $B$ of $Y$ such that $N\subseteq B$. Hence, $p^{-1}(B)$ is $\mi$-negligible, so that $p_*(\abs{\mi})(B)=0$ and $N$ is  $p_*(\abs{\mi})$-negligible. Notice that this argument cannot be entirely disposed of, since the analogous assertion is \emph{false} if we do not assume $\mi$ to be a Radon measure.  Cf.~\cite[p.~30]{Schwartz}  for an illuminating example with $Y$ compact, $N=Y\setminus p(X)$, $\mi$ a probability \emph{Borel} (\emph{yet not Radon}) measure, $p_*(\abs{\mi})$ a probability Radon measure, and $p$  continuous and one-to-one (but $X$ not locally compact, for otherwise $\mi$ should be a Radon measure).} 
\end{oss}

\begin{prop}\label{prop:6}
	Let $X,Y $ be two locally compact spaces with a countable base, and let $\mi$ and $\nu$ be two complex Radon measures on $X$ and $Y$, respectively. Let $p\colon X\to Y$ be a $\mi$-measurable map, and assume that $\nu$ is a weak pseudo-image measure of $\mi$ under $p$. Then, there is a family $(\mi_y)_{y\in Y}$ of complex Radon measures on $X$ such that the following hold:
	\begin{enumerate}
		\item[\textnormal{(1)}] for $\nu$-almost every $y\in Y$, the measure $\mi_y$ is concentrated on $p^{-1}(y)$;
		
		\item[\textnormal{(2)}]  for every  $f\in \Lc^1(X)$, the function $f$ is   $\mi_y$-integrable for $\nu$-almost every $y\in Y$, the mapping $y\mapsto \int_X f\,\dd \mi_y$ is   $\nu$-integrable, and
		\[
		\int_X f\,\dd \mi= \int_Y \int_X f\,\dd \mi_y\,\dd \nu(y).
		\]
		
		\item[\textnormal{(3)}]  for every   positive Borel measurable function $f$ on $X$, the positive function   $y\mapsto \int_X f\,\dd \abs{\mi_y}$ is Borel measurable, and
		\[
		\int_X f\,\dd \abs{\mi}= \int_Y \int_X f\,\dd \abs{\mi_y}\,\dd \abs{\nu}(y).
		\]
		
		\item[\textnormal{(4)}] if we define $A\coloneqq \Set{y\in Y\colon \mi_y\neq 0}$, then $A$ is a Borel subset of $Y$ and $\chi_A\cdot \nu$ is a pseudo-image measure of $\mi$ under $p$.
	\end{enumerate}
	
	Furthermore, if $\nu'$ is a weak pseudo-image measure of $\mi$ under $p$ and $(\mi'_y)$ is a family of complex Radon measures on $X$   such that \textnormal{(1)} and \textnormal{(2)} hold with $\nu$ and $(\mi_y)$ replaced by $\nu'$ and $(\mi'_y)$, respectively, then there are a $\nu'$-measurable subset $A'$ of $Y$ such that $\chi_{A'}\cdot \nu'$ is a pseudo-image measure of $\mi$ under $p$, and a locally $\nu$-integrable complex function $g$ on $Y$  such that   $\chi_{A'}\cdot\nu' = g\cdot \nu$, and $\mi_y = g(y) \mi'_y$   for   $\chi_A\cdot\nu$-almost every $y\in Y$.
\end{prop}

Notice that if $p$ is $\mi$-proper (that is, if $p_*(\abs{\mi})$ is a Radon measure) and $\nu=p_*(\mi)$, then $\mi_y$ is a finite measure and $\mi_y(X)=1$ for $\nu$-almost every $y\in Y$.\footnote{Notice that, in general, $\abs{\nu}\neq p_*(\abs{\mi})$, so that the $\abs{\mi_y}$ need not be   probability measures.}

In addition, observe that the function $g$ in the second part of the statement necessarily vanishes $\nu$-almost everywhere on the complement of $A$, so that the last assertion may be equivalently stated as `$\mi_y = g(y) \mi'_y$   for   $\nu$-almost every $y\in Y$.' We chose to use $\chi_A\cdot \nu$ in order to stress the fact that uniqueness is essentially related to pseudo-image measures (and not simply weak pseudo-image measures).

\begin{proof}
	\textsc{Step I.} Observe first that there are  Borel measurable functions $h_1\colon X\to \T$ and $h_2\colon Y\to \T$ such that $\mi=h_1\cdot \abs{\mi}$ and $\nu=h_2\cdot \abs{\nu}$ (cf.~\cite[Corollary 3 to Theorem 2 of Chapter V, \S\ 5, No.\ 5]{BourbakiInt1}). In addition, observe that there is a positive pseudo-image measure $\nu''$ of $\abs{\mi}$ under $p$, thanks to~\cite[Proposition 1 of Chapter VI, \S\ 3, No.\ 2]{BourbakiInt1}, so that $\nu''$ is absolutely continuous with respect to $\nu$. Consequently, there is a Borel subset $B$ of $Y$ such that $\chi_B\cdot \abs{\nu}$ and $\nu''$ have the same negligible sets, so that also $\chi_B\cdot \abs{\nu}$ is a pseudo-image measure of $\abs{\mi}$ under $p$.
	Therefore, by~\cite[Theorem 2 of Chapter VI, \S\ 3, No.\ 3, and Proposition 2 and Theorem 1 of Chapter V, \S\ 3]{BourbakiInt1} there is a  $\chi_B\cdot \abs{\nu}$-measurable mapping $Y\ni y\mapsto \mi''_y\in \Mc_+(X)$, where $\Mc_+(X)$ denotes the space of positive Radon measures on $X$, endowed with the vague topology, such that (1) and (2) hold with $\abs{\mi}$, $\chi_{B}\cdot \abs{\nu}$, and $(\mi''_y)$ in place of $\mi$, $\nu$, and $(\mi_y)$, respectively, and such that $\mi''_y\neq 0$ for $\chi_B\cdot \abs{\nu}$-almost every $y\in Y$. Then, there is a sequence $(K_j)$ of pairwise disjoint compact subsets of $Y$ such that $Y\setminus \big( \bigcup_j K_j\big)$ is $\chi_B\cdot \abs{\nu}$-negligible and the mapping $K_j\ni y\mapsto \mi''_y\in \Mc_+(X)$ is  continuous for every $j\in \N$. Then, define $\mi_y\coloneqq h_2(y) h_1\cdot  \mi''_y$ if $y\in \bigcup_j K_j$, and $\mi_y\coloneqq 0$ otherwise. 
	Then, (1) and (2) follow. For what concerns (3), since $\abs{\mi_y}=\mi''_y$ for every $y\in \bigcup_j K_j$ and $\abs{\mi_y}=0$ otherwise, it is clear that the mapping $y\mapsto \int_X f\,\dd \abs{\mi_y}$ is Borel measurable for every positive $f\in C_c(X)$, so that the assertion for general Borel measurable functions follows by approximation. Therefore, also (3) follows. Finally, (4) follows from the fact that $A=\Set{y\in Y\colon \int_X 1\,\dd \abs{\mi_y}\neq 0}$, so that $A$ is Borel measurable by (3) and  $\chi_A\cdot \abs{\nu}=\chi_B\cdot \abs{\nu}$ by the previous remarks.
	 
	\textsc{Step II. } Now, assume that $\nu'$ and $(\mi'_y)$ satisfy the assumptions of the statement. The existence of $A'$ may be proved as the existence of $B$ in~\textsc{step I}. Then, $\chi_{A}\cdot \nu'$ and $\chi_{A'}\cdot \nu'$ have the same negligible sets, so that the existence of $g$ follows (cf.~\cite[Corollary 4 to Theorem 2 of Chapter V, \S\ 5, No.\ 5]{BourbakiInt1}). In addition, $g(y)\neq 0$ for $\nu$-almost every $y\in A$. 
	Next, observe that  we may find  a countable dense subset $D$ of $C_c(X)$. Then, for every bounded $\nu$-measurable function $f_2$ on $Y$ and for every $f_1\in D$, the function $(f_2\circ p) f_1$ is $\mi$-integrable, so that 
	\[
	\int_Y f_2(y) \int_X  f_1\,\dd \mi_y\,\dd \nu(y)=\int_X (f_2\circ p) f_1\,\dd \mi= \int_Y f_2(y) \int_X  f_1\,\dd \mi'_y\,\dd \nu'(y)=\int_Y g(y) f_2(y) \int_X  f_1\,\dd \mi'_y\,\dd \nu(y).
	\]
	By the arbitrariness of $f_2$, this proves that
	\[
	\int_X  f_1\,\dd \mi_y = g(y) \int_X  f_1\,\dd \mi'_y
	\]
	for every $f_1\in D$ and for   $\nu$-almost every $y\in Y$. By the arbitrariness of $f_1$, this proves that $\mi_y=g(y) \mi'_y$   for  $\nu$-almost every $y\in Y$.  
\end{proof}

\begin{deff}
	Let $X,Y $ be two locally compact spaces with a countable base, and let $\mi$ and $\nu$ be two complex Radon measures on $X$ and $Y$, respectively. Let $p\colon X\to Y$ be a $\mi$-measurable map, and assume $\nu$ is a weak pseudo-image measure of $\mi$ under $p$. Then, we say that a family $(\mi_y)_{y\in Y}$ of complex Radon measures on $X$ is  a disintegration of $\mi$ relative to   $\nu$ if (1) and (2) of Proposition~\ref{prop:6} hold.  
	
	We say that $(\mi_y)$ is a disintegration of $\mi$ under $p$ if, in addition, $\nu=p_*(\mi)$ (that is, if $p_*(\abs{\mi})$ is a Radon measure and $\mi_y(X)=1$ for $\nu$-almost every $y\in Y$).
\end{deff}

\begin{oss}\label{oss:3}
	Let $X, Y$ be two locally compact spaces with a countable base,   $\mi_1,\mi_2$ two  complex Radon measures on $X$, $\nu$ a complex Radon measures on $Y$, and $p\colon X\to Y$ a $\mi_1$- and $\mi_2$-measurable map such that $\nu$ is a weak pseudo-image measure of both $\mi_1$ and $\mi_2$ under $p$.\footnote{Notice that, if $\nu_1$ and $\nu_2$ are \emph{positive} weak pseudo-image measures of $\mi_1$ and $\mi_2$ under $p$, respectively, then one may choose $\nu=\nu_1+\nu_2$. In other words, this assumption is not restrictive.} 
	Let $(\mi_{j,y})$ be a disintegration of $\mi_j$ relative to its   $\nu$, for $j=1,2$. In addition, let $\mi_{1,y}=\mi_{1,y}^a+\mi_{1,y}^s$ be the Lebesgue decomposition of $\mi_{1,y}$ with respect to $\mi_{2,y}$ (so that $\mi_{1,y}^a$ is absolutely continuous and $\mi_{1,y}^s$ is singular with respect to $\mi_{2,y}$). Analogously, let $\mi_1=\mi_1^a+\mi_1^s$ be the Lebesgue decomposition of $\mi_1$ with respect to $\mi_2$. 	
	Then, $(\mi_{1,y}^a)$ and $(\mi_{1,y}^s)$ are disintegrations of $\mi_1^a$ and $\mi_1^s$, respectively, relative to their common weak pseudo-image measure $\nu$. 
\end{oss}

\begin{proof}
	Take a Borel subset $B$ of $X$ and a locally $\mi_2$-integrable Borel function $f$ on $X$  such that $\mi_1^a=f\cdot \mi_2$ and $\mi_1^s=\chi_B\cdot \mi_1$, so that $\abs{\mi_2}(B)=0$. Since both $\mi_1^a$ and $\mi_1^s$ are absolutely continuous with respect to $\mi_1$, it is clear that $\nu$ is a weak pseudo-image measure of $\mi_1^a$ and $\mi_1^s$ under $p$. Then, by (3) of Proposition~\ref{prop:6}, 
	\[
	0=\abs{\mi_2}(B)=\int_Y \abs{\mi_{2,y}}(B)\,\dd \nu(y),
	\]
	so that $B$ is  $\mi_{2,y}$-negligible for $\nu$-almost every $y\in Y$. Analogously,
	\[
	\int_X g\,\dd \mi_1^a=\int_X g f\,\dd \mi_2=\int_Y \int_X g f\,\dd  \mi_{2,y}\,\dd \nu(y)=\int_Y \int_X g\,\dd (f\cdot \mi_{2,y})\,\dd \nu(y)
	\]
	for every bounded  $\mi$-measurable function $g$ on $X$,	so that $(f\cdot \mi_{2,y})$ is a disintegration of $\mi_1^a$ relative to $\nu$. In a similar way one proves that $(\chi_B\cdot \mi_{1,y})$ is a disintegration of $\mi_1^s$ relative to $\nu$.  Therefore, Proposition~\ref{prop:6} shows that $\mi_{1,y}=f\cdot \mi_{2,y}+\chi_B\cdot \mi_{1,y}$ for $\chi_A\cdot\nu$-almost every $y\in Y$, where $A=\Set{y\in Y\colon \mi_{1,y}\neq 0}$. It then follows that $\mi_{1,y}^a=f\cdot \mi_{2,y}$ and $\mi_{1,y}^s=\chi_B\cdot \mi_{1,y}$ for $\chi_A\cdot\nu$-almost every $y\in Y$. Since both $\mi_{1,y}^s$ and $ \chi_B\cdot \mi_{1,y}$ vanish for every $y\in Y\setminus A$, this is sufficient to prove that $(\mi_{1,y}^s)$ is a disintegration of $\mi_1^s$ relative to $\nu$. Finally, observe that $\chi_A\cdot \nu$ is a weak pseudo-image measure of $\mi_1^a$ under $p$, so that $f\cdot \mi_{2,y}=0$ for $\nu$-almost every $y\in Y\setminus A$. Thus, $(\mi_1^a)$ is a disintegration of $\mi_1^a$ relative to $\nu$.
\end{proof}

\begin{prop}\label{prop:7}
	Let $X, Y$ be two locally compact spaces with a countable base,   $\mi$ a complex    Radon measure on $X$, and $\nu$ a complex Radon measures on $Y$.  Let   $(\mi_y)_{y\in Y}$ be a family of complex Radon measures on $X$ such that
	\begin{equation}\label{eq:2}
		\int_Y \abs{\mi_y}(K)\,\dd \abs{\nu}(y)<+\infty
	\end{equation}
	for every compact subset $K$ of $X$, and such that, for every $f\in C_c(X)$, the mapping $y\mapsto \int_X f\,\dd \mi_y$ is  $\nu$-integrable and
	\begin{equation}\label{eq:1}
		\int_X f\,\dd \mi= \int_Y \int_X f\,\dd \mi_y\,\dd \nu(y).
	\end{equation}
	Assume, in addition, that there is a Borel measurable map $p\colon X\to Y$ such that $\mi_y$ is concentrated on $p^{-1}(y)$ for $\nu$-almost every $y\in Y$.
	Then, $\nu$ is a weak pseudo-image measure of $\mi$ under $p$ and $(\mi_y)$ is a disintegration of $\mi$ relative to  $\nu$.
\end{prop}

This is based on~\cite[Section 2.2]{AleksandrovDoubtsov}.

\begin{proof}
	Let $\Fc$ be the set of $\mi$-integrable functions $f$ on $X$ such that $f$ is $\mi_y$-integrable for $\nu$-almost every $y\in Y$, the mapping $y\mapsto \int_X f\,\dd \mi_y$ is $\nu$-integrable, and~\eqref{eq:1} holds, so that $C_c(X)\subseteq \Fc$ by assumption. Observe that, if $(f_j)$ is a bounded sequence of elements of $\Fc$ which are supported in a fixed compact subset $K$ of $X$ and converge pointwise to some $f$ on $X$, then $f\in \Fc$ by dominated convergence, thanks to~\eqref{eq:2}. Thus, by transfinite induction, $\Fc$ contains all bounded Borel measurable functions with compact support. 
	
	Now, let $h_1\colon X\to \T$ and $h_2\colon Y\to \T$ be  two Borel measurable functions   such that $\mi=h_1\cdot \abs{\mi}$ and $\nu=h_2 \cdot \abs{\nu}$ (cf.~\cite[Corollary 3 to Theorem 2 of Chapter V, \S\ 5, No.\ 5]{BourbakiInt1}).  Then, for every bounded Borel function $f$ on $X$ with compact support,
	\[
	\int_X f \,\dd \abs{\mi}=\int_X f \overline{h_1}\,\dd \mi= \int_Y \int_X f \overline{h_1} \,\dd \mi_y\,\dd \nu(y)=\int_Y \int_X f \,\dd ( h_2(y) \overline{h_1}\cdot \mi_y)\,\dd \abs{\nu}(y).
	\]
	In other words, we may reduce to the case in which both $\mi$ and $\nu$ are positive Radon measures. Let us then prove that $\mi_y$ is positive for   almost every $y\in Y$. In order to do so, take a countable dense subset $D$ of the cone of positive elements of $C_c(X)$ and observe that, if $f\in D$ and $g$ is a positive bounded Borel measurable function on $Y$, then 
	\[
	0\meg\int_X f (g\circ p)\,\dd \mi= \int_Y g(y) \int_X f\,\dd \mi_y\,\dd \nu(y).
	\]
	The arbitrariness of $g$ then shows that $ \int_X f\,\dd \mi_y\Meg 0$ for  $\nu$-almost every $y\in Y$. The arbitrariness of $f \in D$ then shows that $\mi_y$ is positive for $\nu$-almost every $y\in Y$.
	
	Now, if $N$ is a   relatively compact $\mi$-negligible subset of $X$, then there is a relatively compact $\mi$-negligible Borel subset $B$ of $X$ such that $N\subseteq B$, so that 
	\[
	0=\mi(B)=\int_Y \mi_y(B)\,\dd \nu(y),
	\]
	whence $\mi_y(B)=0$ for locally $\nu$-almost every $y\in Y$. Consequently, $N$ is $\mi_y$-negligible for locally $\nu$-almost every $y\in Y$. Since $X$ is the union of a sequence of compact subsets, the same holds for every $\mi$-negligible subset $N$ of $X$, relatively compact or not. In particular, if $f$ is a $\mi$-measurable function on $X$, then $f$ is $\mi_y$-measurable for locally $\nu$-almost every $y\in Y$; in addition, if $f$ is bounded and has compact support, then~\eqref{eq:1} holds, that is, $f\in \Fc$. 
	
	Now, take a $\mi$-integrable positive function $f$ on $X$. Then, there is an increasing sequence $(f_j)$ of bounded $\nu$-measurable functions with compact support on $X$ which converges pointwise to $f$. Since $f_j\in \Fc$ for every $j\in\N$, by monotone convergence we see that $f\in \Fc$. In particular, $f$ is $\mi_y$-integrable for locally $\nu$-almost every $y\in Y$. To see that $\nu$ is a pseudo-image measure of $\mi$ under $p$, it then suffices to take a $\nu$-negligible Borel subset $N$ of $Y$ and to observe that $p^{-1}(N)$ is a Borel subset of $X$, so that
	\[
	\mi(p^{-1}(N))=\int_N \mi_y(X)\,\dd \nu(y)=0,
	\]
	whence the result by the arbitrariness of $N$.
\end{proof}


\begin{thebibliography}{99} 
	\bibitem{Aleksandrov}
	Aleksandrov, A.\ B., Existence of Inner Functions in the Unit Ball, \emph{Mat.\ Sib.} \textbf{118} (1982), p.~147--163.
	
	\bibitem{Aleksandrov1}
	Aleksandrov, A.\ B., Multiplicity of Boundary Values of Inner Functions, \emph{Izv.\ Akad.\ Nauk Armyan.\ SSR Ser.\ Math.} \textbf{22} (1987), p.~490--503, 515.
	
	\bibitem{Aleksandrov2}
	Aleksandrov, A.\ B., Inner Functions and Related Spaces of Pseudocontinuable Functions, \emph{Zap.\ Nauchn.\ Sem.\ Leningrad.\ Otdel.\ Mat.\ Inst.\ Steklov.\ (LOMI)} \textbf{170} (1989), \emph{Issled.\ Linein Oper.\ Teorii Funktsii.} 17, p.~7--33, 321; translations in \emph{J.\ Soviet Math.} \textbf{63} (1993), p.~115--159.
	
	\bibitem{Aleksandrov3}
	Aleksandrov, A.\ B., On the Maximum Principle for Pseudocontinuable Functions,  \emph{Zap.\ Nauchn.\ Sem.\ S.-Peterburg.\ Otdel.\ Mat.\ Inst.\ Steklov. (POMI)} \textbf{217} (1994),  \emph{Issled.\ po Linein Oper.\ i Teor.\ Funktsii.} 22, p.~16--25, 218; translation in \emph{J.\ Math.\ Sci.\ (New York)} \textbf{85} (1997), p.~1767--1772.
	
	\bibitem{Aleksandrov4}
	Aleksandrov, A.\ B., On the Existence of Angular Boundary Values for Pseudocontinuable Functions, \emph{Zap.\ Nauchn.\ Sem.\ S.-Peterburg.\ Otdel.\ Mat.\ Inst.\ Steklov. (POMI)} \textbf{222} (1995),  \emph{Issled.\ po Linein Oper.\ i Teor.\ Funktsii.} 23, p.~5--17, 307; translation in \emph{J.\ Math.\ Soc.\ New York} \textbf{87} (1997), p.~3871--3787.
	
	\bibitem{Aleksandrov5}
	Aleksandrov, A.\ B., Isometric Embeddings of Coinvariant Subspaces of the Shift Operator,  \emph{Zap.\ Nauchn.\ Sem.\ S.-Peterburg.\ Otdel.\ Mat.\ Inst.\ Steklov. (POMI)} \textbf{232} (1996),  \emph{Issled.\ po Linein Oper.\ i Teor.\ Funktsii.} 24, p.~5--15, 213; translation in \emph{J.\ Math.\ Sci.\ (New York)} \textbf{92} (1998), p.~3543--3549.
	
	
	\bibitem{AleksandrovDoubtsov}
	Aleksandrov, A.\ B., Doubtsov, E., Clark Measures on the Complex Spheres, \emph{J.\ Funct.\ Anal.} \textbf{278} (2020), 108314.
	
	\bibitem{AleksandrovDoubtsov2}
	Aleksandrov, A.\ B., Doubtsov, E., Clark Measures and de Branges--Rovnyak Spaces in Several Variables, \emph{Complex Var.\ Elliptic} \textbf{68} (2023), p.~212--221. 
	
	\bibitem{BesicovitchMoran}
	Besicovitch, A.\ S., Moran, P.\ A.\ P., The Measure of Product and Cylinder Sets, \emph{J.\ London Math.\ Soc.} \textbf{20} (1945), p.~110--120. 
	
	\bibitem{BourbakiInt1}
	Bourbaki, N., \emph{Integration, I}, Springer, 2004. 
	
	\bibitem{Calzi}
	Calzi, M., Positive Pluriharmonic Functions on Symmetric Siegel Domains, preprint (2024), arXiv:2403.05429.
	
	\bibitem{CimaMathesonRoss}
	Cima, J.\ A., Matheson, A.\ L., Ross, W.\ T., \emph{The Cauchy Transform}, AMS,  2006.
	
	\bibitem{Clark}
	Clark, D.\ N., One Dimensional Perturbation of Restricted Shifts, \emph{J.\ Analyse Math.} \textbf{25} (1972), p.~169--191.
	
	\bibitem{FarautKoranyi}
	Faraut, J., Kor\'anyi, A., Function Spaces and Reproducing Kernels on Bounded Symmetric Domains, \emph{J.\ Funct.\ Anal.} \textbf{88} (1990), pp.~64--89.
	 
	
	\bibitem{Jury1}
	Jury, M.\ T., Operator-Valued Herglotz Kernels and Functions of Positive Real Part on the Ball, \emph{Complex Anal.\ Oper.\ Theory} \textbf{4} (2010), p.~301--317.
	
	\bibitem{Jury2}
	Jury, M.\ T., Clark Theory in the Drury--Arveson Space, \emph{J.\ Funct.\ Anal.} \textbf{266} (2014), p.~3855--3893.
	
	\bibitem{JuryMartin}
	Jury, M.\ T., Martin, R.\ T.\ W., Aleksandrov--Clark Theory for Drury--Arveson Space,  \emph{Integr.\ Equ.\ Oper.\ Theory} \textbf{90} (2018), 45.
	
	\bibitem{Kai}
	Kai, C., A Characterization of Symmetric Siegel Domains by Convexity of CayleyTransform Images, \emph{Tohoku Math.\ J.} \textbf{59} (2007), p.~101--118.
	
	\bibitem{Koranyi}
	Kor\'anyi, A., The Poisson Integral for Generalized Half-Planes and Bounded Symmetric Domains, \emph{Ann.\ Math.} \textbf{82} (1965), pp.~332--350.
	
	\bibitem{Koranyi2}
	Kor\'anyi, A., Poisson Integrals and Boundary Components of Symmetric Spaces, \emph{Inventiones Math.} \textbf{34} (1976), p.~19--35.
	
	\bibitem{KoranyiVagi}
	Kor\'anyi, A., V\'agi, S., Rational Inner Functions on Bounded Symmetric Domains, \emph{Trans.\ Amer.\ Math.\ Soc.} \textbf{254} (1979), p.~179--193.
	
	\bibitem{KoranyiPukanszki}
	Kor\'anyi, A., Puk\'anszky, L., Holomorphic Functions with Positive Real Part on Polycylinders, \emph{Trans.\ Amer.\ Math.\ Soc.} \textbf{108} (1963), pp.~449--456.
	 
	\bibitem{Krantz}
	Krantz, S. G., \emph{Function Theory of Several Complex Variables}, Amer.\ Math.\ Soc.\ Chelsea Publ., 2001.
	
	\bibitem{Loos}
	Loos, O., \emph{Bounded Symmetric Domains and Jordan Pairs}, University of California, Irvine, 1977.
	 
	\bibitem{MackeyMellon}
	Mackey, M., Mellon, P., Angular Derivatives on Bounded Symmetric Domains, \emph{Isr.\ J.\ Math.} \textbf{138} (2003), p.~291--315.
	
	\bibitem{Poltoratski2}
	Poltoratski\u{\i}, A.\ G., Boundary Behavior of Pseudocontinuable Functions, \emph{Algebra i Analiz} \textbf{5} (1993), p.~189--210; English transl.: \emph{St.\ Petersburg Math.\ J.} \textbf{5} (1994), p.~389--406.
	
	\bibitem{Poltoratski}
	Poltoratski, A.\ G., On the Distributions of Boundary Values of Cauchy Integrals, \emph{Proc.\ Amer.\ Math.\ Soc.} \textbf{124 } (1996), p.~2455--2463.
	
	\bibitem{PoltoratskiSarason}
	Poltoratski, S.\ G., Sarason, D., Aleksandrov--Clark Measures, \emph{Contemp.\ Math.} \textbf{393} (2006), p.~1--14.
	 
	\bibitem{Rudin}
	Rudin, W., \emph{Function Theory on the Unit Ball of $\C^n$}, Springer, 1980.
	
	\bibitem{Schwartz}
	Schwartz, L., \emph{Radon Measures on Arbitrary Topological Spaces and Cylindrical Measures}, Oxford University Press, 1973.
	
	\bibitem{Shabat}
	 Shabat, B.\ V.,  \emph{Introduction to Complex Analysis. Part II: Functions of Several Variables}, Translations of Mathematical Monographs, vol.\ 110, American Mathematical Society, Providence, RI, 1992, translated from the third (1985) Russian edition by J.\ S.\ Joel.
	
	\bibitem{SteinWeiss}
	Stein, E.\ M., Weiss, N.\ J., On the Convergence of Poisson Integrals, \emph{Trans.\ Amer.\ Math.\ Soc.} \textbf{140} (1969), pp.~35--54.
	
	\bibitem{Timotin}
	Timotin, D., A Short Introduction to de Branges--Rovnyak Spaces, in: Invariant Subspaces of the Shift Operator, \emph{Contemp.\ Math.} \textbf{638} (2015), p.~21--38. 
\end{thebibliography}
\end{document}